\documentclass[10pt]{article}

\usepackage[utf8]{inputenc}

\usepackage{hyperref}
\usepackage{verbatim} 

\usepackage{graphicx}
\usepackage{caption}
\usepackage{subcaption}
\usepackage{amssymb}
\usepackage{ textcomp } 
\usepackage{color}
\usepackage{enumerate}

\usepackage{amsopn}

\usepackage{amsmath}
\usepackage{amsthm}

\usepackage{multicol}
\usepackage[affil-it]{authblk}

\usepackage[ruled,vlined]{algorithm2e}

\newcommand{\R}{\mathbb{R}}

\newcommand{\calS}{\mathcal{S}}
\newcommand{\calP}{\mathcal{P}}
\newcommand{\calA}{\mathcal{A}} 
\newcommand{\calD}{\mathcal{D}}
\newcommand{\calC}{\mathcal{C}}
\newcommand{\calH}{\mathcal{H}}
\newcommand{\calR}{\mathcal{R}}

\newcommand{\abs}[1]{\left\vert #1 \right\vert}
\newcommand{\norm}[1]{\left\Vert #1 \right\Vert}
\newcommand{\set}[1]{\left\lbrace #1\right\rbrace}
\newcommand{\sse}{\subseteq}
\newcommand{\sprod}[1]{\left\langle #1 \right\rangle}

\newcommand{\sph}{\mathbb{S}}

\newcommand{\prb}[1]{\mathbb{P}\left( #1 \right)}
\newcommand{\erw}[1]{\mathbb{E}\left( #1 \right)}

\usepackage{dsfont}

\DeclareMathOperator{\conv}{conv}
\DeclareMathOperator{\cone}{cone}

\DeclareMathOperator{\relint}{relint}

\DeclareMathOperator{\dist}{dist}

\DeclareMathOperator{\sgn}{sgn}
\DeclareMathOperator{\supp}{supp}

\DeclareMathOperator{\argmin}{argmin}

\DeclareMathOperator{\pos}{pos}

\newcommand{\gr}{\mathbb{G}}

\newtheorem{lem}{Lemma}
\newtheorem{prop}[lem]{Proposition}

\newtheorem{theo}[lem]{Theorem}
\newtheorem{cor}[lem]{Corollary}
\newtheorem{example}[lem]{Example}
\newtheorem{defi}[lem]{Definition}

\theoremstyle{remark}

\newtheorem{rem}[lem]{Remark}

\title{A Geometrical Stability Condition for Compressed Sensing \footnote{This article has been accepted for publication in Linear Algebra and its Applications. It can be found via its DOI \href{http://dx.doi.org/10.1016/j.laa.2016.04.017}{\tt 10.1016/j.laa.2016.04.017}}}

\author{Axel Flinth\thanks{ E-mail: flinth@math.tu-berlin.de}}
\affil{Institut für Mathematik \\ Technische Universität Berlin}

\begin{document}

\maketitle

\begin{abstract}
During the last decade, the paradigm of compressed sensing has gained significant importance in the signal processing community. While the original idea was to utilize sparsity assumptions to design powerful recovery algorithms of vectors $x \in \R^d$, the concept has been extended to cover many other types of problems. A noteable example is low-rank matrix recovery. Many methods used for recovery rely on solving convex programs.

A particularly nice trait of compressed sensing is its geometrical intuition. In recent papers, a classical optimality condition has been used together with tools from convex geometry and probability theory to prove beautiful results concerning the recovery of signals from Gaussian measurements. In this paper, we aim to formulate a geometrical condition for stability and robustness, i.e. for the recovery of approximately structured signals from noisy measurements.

We will investigate the connection between the new condition with the notion of \emph{restricted singular values}, classical stability  and robustness conditions in compressed sensing, and also to important geometrical concepts from complexity theory. We will also prove the maybe somewhat surprising fact that for many convex programs, exact recovery of a signal $x_0$ immediately implies some stability and robustness when recovering signals close to $x_0$.



\end{abstract}

\section{Introduction}

Suppose that we are given linear measurements $b \in \R^m$ of a signal $x_0 \in \R^d$, i.e. $b = Ax_0$ for  some matrix $A \in \R^{m,d}$, and are asked to recover the signal from them. If $d>m$, this will not be trivial, since the map $x_0 \mapsto b$ in that case won't be injective. If, however, one assumes that $x_0$ in some sense is \emph{sparse}, e.g., that many of $x_0$'s entries vanish, we can still recover the signal, e.g. with the help of $\ell_1$-minimization \cite{CandesRombergTao2006}:
\begin{align}
	\min \norm{x}_1 \text{ subject to } Ax=b \tag{$\calP_1$}
\end{align}
This is the philosophy of \emph{compressed sensing}, an area of mathematics which has achieved major attention over the last decade. It has become a standard technique to choose $A$ at random, and then to ask the question how large the number of measurements $m$ has to be in order for $(\calP_1)$ to be successful with high probability. A popular assumption is that $A$ has the Gaussian distribution, i.e., that the entries are i.i.d. standard normally distributed. 

A widely used criterion to ensure that $(\calP_1)$ is successful is the \emph{$RIP$-property}. Put a bit informally, a matrix $A$ is said to possess the $RIP$-property if its $RIP$-constants;
\begin{align*}
	\delta_k = \min  \set{\delta >0 :\forall x \text{ k-sparse } : (1-\delta)\norm{x}_2^2 \leq \norm{Ax}_2^2 \leq (1+\delta)\norm{x}_2^2 },
\end{align*}
are small.

The idea of using convex programs like $(\calP_1)$ to recover structured signals has come to be used in a much wider sense than the one above. Some examples of structure assumptions that have been considered in the literature are dictionary sparsity \cite{davenport2013signal},  block sparsity \cite{stojnic2009reconstruction}, sparsity with prior information \cite{FriedMansSaabYil2012,XuetAl2009,OymakHassibiEtAl2012}, saturated vectors (i.e. vectors with $\abs{x(i)}=\norm{x}_\infty$ for many $i$) \cite{fuchs2011} and low-rank assumptions for matrix completion \cite{candes2010matrix}. Although these problems may seem very different at first sight, they can all be solved with the help of a convex program of the form 
\begin{align}
	\min f(x) \text{ subject to } Ax=b, \tag{$\calP_f$}
\end{align}
where $f$ is some convex function defined on an appropriate space. In all of the mentioned examples above, $f$ is chosen to be a norm, but this is not per se necessary.

The connection between the different convex program approaches was thoroughly investigated in  \cite{chandrasekaran2012convex}, in which the very general case of $f$ being an \emph{atomic norm} was investigated. An atomic norm $\norm{\cdot}_{\calA}$ is thereby the \emph{gauge function} of the convex hull of a compact set $\calA$;
\begin{align*}
	\norm{x}_\calA = \sup \set{t>0 \ \vert \ t x \in \conv \calA}.
\end{align*}
The authors derive bounds on how many Gaussian measurements are required for a convex optimization problem of the form $(\calP_f)$ to be able to recover a signal $x_0$ with high probability.  Their arguments have a geometrical flavor, since they utilize a well known optimality condition regarding the \emph{descent cone} $\calD(\norm{\cdot}_\calA, x_0)$ (see Definition \ref{def:descentCone} and Lemma \ref{lem:descConeKernel} below) of $\norm{\cdot}_\calA$. Other important theoretical tools are \emph{Gaussian widths} and Gordon's \emph{escape through a mesh lemma} \cite{Gordon1988}, which will be discussed in Section \ref{sec:GaussMeas} of this article. The authors of \cite{AmelunxLotzMcCoyTropp2014} make a similar analysis for even more general functions $f$, only assuming that they are convex. They use the so called \emph{statistical dimension} of the descent cone  for determining the threshold value of measurements. 

\subsection*{The Problem of Stability and Robustness}
In applications, the linear measurements $b$ are often contaminated with noise. This means that we are actually given data $b = Ax+n$, where $n$ is a noise vector. A popular assumption is that $n$ is bounded in $\ell_2$-norm, i.e. $\norm{n}_2 \leq \epsilon$. Moreover, it is not often not entirely realistic to assume that the signal $x_0$ is exactly sparse (or more generally does not exactly have the structure assumed by the model), but rather that the distance to the set of sparse (structured) signals $\calC$ is small, i.e.
\begin{align}
	\dist_{\norm{\cdot}_*}(x_0, \calC) = \inf_{c \in \calC} \norm{x_0 -c}_*,
\end{align}
where $\norm{\cdot}_*$ is some norm. If the above quantity is small for a vector $x_0$, we will call it \emph{approximately structured}.

There are several approaches to approximately recover approximately structured signals from noisy measurements -- the one we will consider is the following \emph{regularized convex optimization problem:}
\begin{align*}
	\min f(x) \text{ subject to } \norm{Ax-b}_2 \leq \epsilon. \tag{$\calP_f^\epsilon$}
\end{align*}
This approach was investigated already in the earliest works on compressed sensing, where
 $f = \norm{\cdot}_1$ and $\calC = \calS_k =\set{k\text{-sparse signals}}$ \cite{CandesRombergTao2006}. In short, it turns out that (somewhat stronger) assumptions on the $RIP$-constants suffice to prove that the solution $x^*$ of $(\calP_f^\epsilon)$ in this case satisfies a bound of the form
\begin{align*}
	 \norm{x-x^*}_2 \leq \kappa_1 \dist_{\norm{\cdot}_1} (x_0, \calS_k) + \kappa_2 \epsilon,
\end{align*}
where $\kappa_1$ and $\kappa_2$ depend on the $RIP$-constants of $A$. In this paper, we will formulate and investigate a geometrical criterion for general convex programs $(\calP_f^\epsilon)$ to satisfy such a bound, at least when $f$ is a norm. We will call such programs \emph{robust} (with respect to noise) and \emph{stable} (with respect to distance of $x_0$ to $\calC$). The criterion, which we will call the \emph{Angular Separation Criterion} or $ASC$, will only depend on the relative positions of the kernel of $A$ and the descent cone $\calD(f, x_0)$. More specifically, we will prove that if the mentioned sets have a positive angular separation, $(\calP_f^\epsilon)$ will be robust and, in the case of $f$ being a norm, stable. 

We will furthermore relate the $ASC$ to known criteria for stability and robustness from the literature. We will also prove the somewhat remarkable fact that for a very large class of norms, the $ASC$ is in fact \emph{implied} by the ability of $(\calP_f)$ to recover signals exactly from \emph{noiseless} measurements. The idea originates from a part of the Master Thesis \cite{Flinth2015Master} of the author.

\subsection*{Related Work and Contributions of This Paper}
	Some research towards a geometrical understanding of the stability of compressed sensing has already been conducted. Here we list a few of the approaches that have been considered. First, we would like to mention the so called $RIP$-$NSP$-condition from the recent paper \cite{cahill2015gap}. This paper only deals with the classical compressed sensing setting, i.e., that the signal of interest is approximately sparse and $\ell_1$-minimization is used to recover it. A matrix $A$ is said to satisfy the $RIP$-$NSP$-condition if there exists another matrix $B$, which has the $RIP$-property, such that $\ker A = \ker B$. The authors of \cite{cahill2015gap} prove that this is enough to secure stability and robustness of $(\calP_1^\epsilon)$. This is intriguing, as it shows that stability and robustness can be secured by only considering the kernel of the measuring matrix.
	
	Another line of research is the so called \emph{Robust Width Property}, which was developed in \cite{cahill2014robwidth}. The authors of said article define \emph{compressed sensing spaces}, a general framework that covers both different types of sparsity in $\R^d$ as well as the case of low rank matrices. The main part of the definition of a compressed sensing space is a norm decomposability condition; if $\norm{\cdot}$ is the norm induced by the inner product of the Hilbert space $\calH$, and $\norm{\cdot}_*$ is another norm on $\calH$, $(\calH, \calC,\norm{\cdot}_{*})$ is said to be a compressed sensing space with bound $L$ if for every $a$ in the subset $\calC\sse \calH$ and $z \in \calH$, there exists a decomposition $z=z_1 +z_2$ such that
	\begin{align*}
		\norm{a+z}_{*} = \norm{a}_{*} + \norm{z_1}_{*} \ , \ \norm{z_2}_{*} \leq L \norm{z}.
	\end{align*}
	A matrix is then said to satisfy the $(\rho, \alpha)$-\emph{robust width property} if for every $z \in \calH$ with $\norm{Az}\leq \alpha \norm{z}$, we have $\norm{z}\leq \rho \norm{z}_{*}$. This robust width property is in fact equivalent with the stability and robustness of $(\calP_f^\epsilon)$ with $f(x) = \norm{x}_{*}$. One large difference between this approach and ours is that we do not require any norm decomposability conditions. For instance, as was pointed out in \cite{cahill2015gap}, the $\ell_\infty$-norm in $\R^d$ is not well suited to be included in this framework.
	
	As for the problem of stability of recovering almost sparse vectors, we want to mention the paper \cite{xu2011precise}. The authors of that article carry out an asymptotic analysis of the threshold amount of Gaussian measurements needed for the classical technique of $\ell_1$-minimization for sparse recovery to be stable. The analysis heavily relies on the theory of so called \emph{Grassmannian angles} of polytopes, which is a purely geometrical concept. This approach has connections to, but is still relatively far away from, the one in this work. In particular, it is by no means straight-forward, if at all possible, to generalize it other problems than $\ell_1$-minimization.
	
	The condition presented in this paper is highly related to several other geometric stability measures: so called \emph{restricted singular values} \cite{amelunxen2014gordon} of matrices on cones, as also \emph{Renegar's condition number} \cite{BelloniFreund2007,renegar1992} as well as the \emph{Grassmannian condition number} \cite{amelunxenBurgisser2012}  (see Section \ref{sec:ASCvsConds}). In fact, already in the previously mentioned paper \cite{chandrasekaran2012convex}, it is proven that if the smallest singular value of $A$ restricted to the descent cone of the functional $f$ does not vanish, we will have stability. During the final review of this paper, the author was made aware of the fact that also the connection between Renegar's Condition number and robustness of compressed sensing has recently been investigated in \cite{Roulet2015Renegar}. 
	
	This work provides a new, more elementary, perspective to the above mentioned notions, and in particular establishes its relations to classical criteria for stability in compressed sensing. Another contribution of this paper is the observation that if $f$ is a norm, the criterion also implies stability, an observation which, to the best of the knowledge of the author, has not been done before. 
	
	\subsection*{Notation}
	Throughout the whole paper, $\calH$ will denote a general, finite-dimensional Hilbert space. The corresponding inner product will be denoted $\sprod{\cdot, \cdot}$, and the induced norm by $\norm{\cdot}$. For a subspace $U \sse \calH$, we will write
	\begin{align*}
		\norm{x +U} = \inf_{u \in U} \norm{x+u}.
\end{align*}		 
Note that due to the finite-dimensionality of $\calH$, this infimum is in fact attained.

In the final part of the paper, we will deal with Gaussian vectors and Gaussian matrices. A random vector, or linear map, is said to be Gaussian if its representation in an orthonormal basis, or in a pair of such, has i.i.d. standard normally distributed entries.

The entries of a vector $x$ in $\R^d$ will be denoted $x(1), x(2), \dots, x(d)$.  $\R^d$ is equipped with the standard scalar product whose induced norm is the $\ell_2$-norm:
\begin{align*}
	\sprod{x,y} = \sum_{i=1}^d x(i)y(i), \quad \norm{x}_2 = \sqrt{\sum_{i=1}^d x(i)^2}.
\end{align*}
	




\section{A Geometrical Robustness Condition} \label{sec:geomStab}

Let us begin by considering the most classical compressed sensing setting: that is, the problem of retrieving a signal $x_0 \in \R^d$  with few non-zero entries from exact measurements $b = Ax_0$ using $\ell_1$-minimization (i.e. $(\calP_1)$).
One of the most well known criteria for $(\calP_1)$ to be successful is the so called Null Space Property, $NSP$. A matrix $A$ satisfies the $NSP$ with respect to the index set $S_0 \sse \set{1, \dots d}$ if for every $\eta \in \ker A$, we have
\begin{align} \label{eq:NSP}
	\sum_{i \in S_0} \abs{\eta(i)} < \sum_{i \notin S_0} \abs{\eta(i)}.
\end{align}
The $NSP$ with respect to $S$ is in fact equivalent to $(\calP_1)$ recovering a signal $x_0$ with support $S$ \cite{MathIntroToCS}. The $NSP$ has a geometrical meaning. In order to explain it, let us first define \emph{descent cones} of  a function $f$.

\begin{defi} \label{def:descentCone}
	Let $\calH$ be a finite-dimensional Hilbert space, and $f: \calH \to \R$. The \emph{descent cone} $\calD(f,x_0)$ of $f$ at the point $x_0 \in \calH$ is the cone generated by the descent directions of $f$ at $x_0$, i.e. 
	\begin{align*}
		\calD(f,x_0) = \set{\eta \ \vert \ \exists \tau>0 : f(x_0 + \tau \eta) \leq f(x_0)}.
	\end{align*}
\end{defi}

\begin{example} \label{ex:l1desc}
	The descent cone of the $\ell_1$-norm at a vector $x_0$ supported on the set $S_0$ is given by
	\begin{align*}
		\calD(\norm{\cdot}_1,x_0)= \Big\{\eta \ \vert \ \sum_{i \notin S_0} \abs{\eta}_i \leq -\sum_{i \in S_0} \sgn(x_0(i)) \eta_i \Big\}.
	\end{align*}
	\underline{Proof:} Since the conditions for $\eta$ to belong to both the left hand and the right hand set, respectively, is invariant under scaling, we may assume that $\eta$ has small norm. Then we have
	\begin{align*}
		\abs{x_0(i)+\eta(i)} = \begin{cases} \abs{x_0(i)}+\sgn(x_0(i))\eta_i & \ i \in S_0 \\
											\abs{\eta(i)} & \ i \notin S_0 \end{cases},
	\end{align*}
	which implies $\norm{x_0+\eta}_1 - \norm{x_0} = \sum_{i \notin S_0} \abs{\eta}_i + \sum_{i \in S_0} \sgn(x_0(i)) \eta_i$. This is smaller than or equal to zero exactly when $\sum_{i \notin S_0} \abs{\eta}_i  \leq - \sum_{i \in S_0} \sgn(x_0(i)) \eta_i$
\end{example}

With the last example in mind, it is not hard to convince oneself that Equation \eqref{eq:NSP} exactly states that the vector $\eta$ does not lie in the descent cone of any signal supported on the set $S_0$. I.e., the $NSP$ actually reads
\begin{align*}
	\forall x_0: \supp x_0 = S_0: \calD(\norm{\cdot}_1, x_0) \cap \ker A = \set{0}.
\end{align*}
This observation can be generalized to more general situations, as the following well-known lemma shows.
\begin{lem}\label{lem:descConeKernel} (E.g. \cite[Proposition 2.1]{chandrasekaran2012convex}.)
Let $f: \calH \to \R$ be convex, $A : \calH \to \R^m$ be linear and consider the program $(\calP_f)$, with $b=Ax_0$ for noiseless recovery of the signal $x_0$. The solution of $(\calP_f)$ is equal to $x_0$ if and only if
\begin{align}
	\calD(f, x_0) \cap \ker A = \set{0}. \label{eq:geomExactCond}
\end{align}
\end{lem}

\begin{figure}
\centering
\includegraphics[scale=.3]{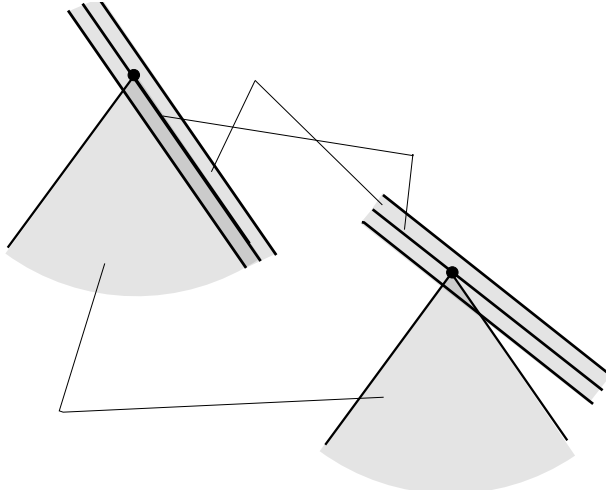}
\caption{ The impact of an angular separation of $\calD(f,x_0)$ and $\ker A$. Note that locally, the sets $x_0 +\calD(f,x_0)$ and $\set{x \vert f(x) \leq f(x_0)}$ have the same structure. \label{fig:angsep}}
\end{figure}

Can we use the previous lemma to develop a geometrical intuition of what we have to assume in order to prove stability and robustness of the recovery using $(\calP_f^\epsilon)$? The main difference between the program $(\calP_f)$ and $(\calP_f^\epsilon)$ is that the former is only allowed to search for a solution in the set $x_0 +\ker A$, while the latter can search in a tubular neighborhood of the same set. Figure \ref{fig:angsep} suggests that if the descent cone $\calD(f, x_0)$ and $\ker A$ do not only trivially intersect each other, but also have an angular separation, it should be possible to prove that the intersection of the mentioned tubular neighborhood and the set $\set{x \vert f(x) \leq f(x_0)}$ is not large. This should in turn imply robustness. (In fact, this intuition was used already in \cite{CandesRombergTao2006} when proving robustness in the original compressed sensing setting.) To provide a precise formulation of angular separation, we first define the \emph{$\theta$-expansion of a cone}.

\begin{defi} \label{def:muExp}
Let $\calH$ be a finite-dimensional Hilbert space with norm $\norm{\cdot}$ and scalar product $\sprod{\cdot,\cdot}$. Let further $C\sse \calH$ be a convex cone, i.e. a convex set with $\tau C=C$ for every $\tau>0$. Then for $\theta \in [0, \pi]$, we define the $\theta$-expansion $C^{\wedge \theta}$ as the set
\begin{align*}
	C^{\wedge \theta} = \set{x \in \calH \ \big\vert \ \exists y \in C : \sprod{x,y} \geq \cos(\theta) \norm{x} \norm{y}}.
\end{align*} 
\end{defi}
For an illustration of the relation between a cone $C$ and its $\theta$-expansion $C^{\wedge \theta}$, see Figure \ref{fig:thetaExp}. Before moving on, let us make some remarks.

\begin{figure}
\centering
	\includegraphics[scale=.45]{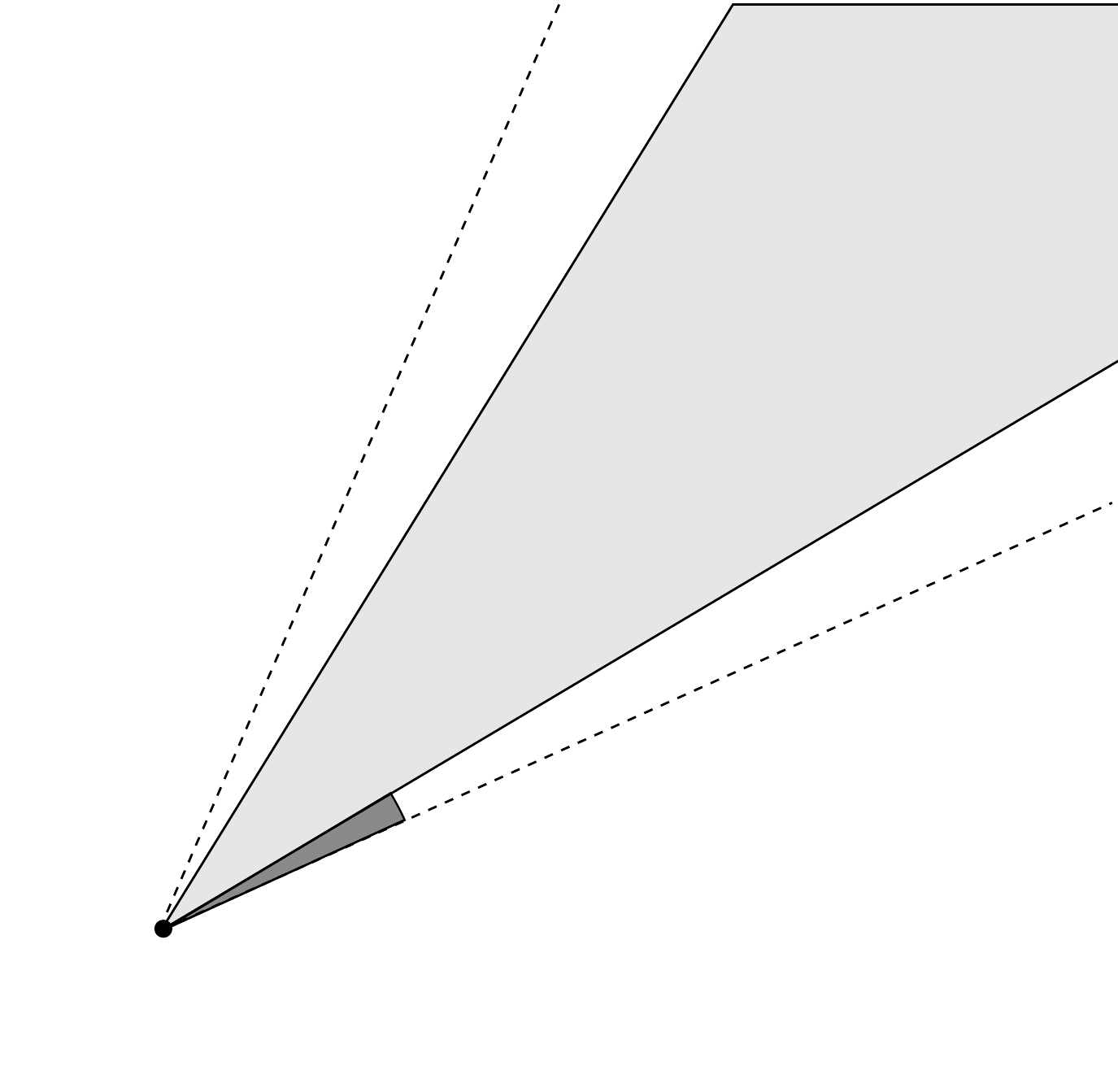}
	\caption{A convex cone $C$ and its $\theta$-expansion $C^{\wedge \theta}$. \label{fig:thetaExp}}
\end{figure}

\begin{rem}
\begin{enumerate}[(i)]
\item $C = C^{\wedge 0} \sse C^{\wedge \theta}$ for every $\theta >0$.

\item If $C$ is closed, $C^{\wedge \theta}$ can alternatively defined as the set
\begin{align*}
	C^{\wedge \theta}=\cone (\{x \in \calH \ \big\vert  \ \norm{x}=1, \sup_{y \in C, \norm{y}=1}  \sprod{x,y} \geq \cos(\theta)\} ).
\end{align*}

\item $C^{\wedge \theta}$ is always a cone, but not always convex. As a concrete counterexample, consider the closed, convex cone $K \sse \R^3$
	\begin{align*}
		K= \cone \left(\set{x \in \R^3 \ \big\vert \ x(1)=1, x(2)=0, \abs{x(3)} \leq 1}\right).
\end{align*}	 
	Using the previous remark, we can calculate $K^{\wedge \theta}$ exactly. We have for $x\in \R^3$
	\begin{align*}
		\sup_{y \in K, \norm{x}=1} \sprod{x,y} = \sup_{t\in [-1,1]} \frac{x(1) + t x(3)}{\sqrt{1+t^2}} = \begin{cases} \sqrt{x(1)^2 + x(3)^2} &\text{if } \abs{x(3)}\leq x(1) \\  
		\frac{x(1) + \abs{x(3)}}{\sqrt{2}}& \text{else,}
		 \end{cases}
	\end{align*}
	where the last equality can be proven using elementary calculus. The last remark now tells us that $K^{\wedge \frac{\pi}{6}}$ is given by
	\begin{align*}
		\set{ y\in \R^3 \ \big\vert \ \frac{\sqrt{3}\norm{y}_2}{2} \leq \begin{cases} \sqrt{y(1)^2 + y(3)^2} &\text{if } \abs{y(3)}\leq y(1) \\  
		\frac{y(1) + \abs{y(3)}}{\sqrt{2}}& \text{else.} \end{cases}}.
	\end{align*}
	This set is not convex; for instance, the points $y_1=\left(1,\sqrt{\frac{2}{3}},1\right)$ and $y_2=\left(1,\sqrt{\frac{2}{3}},-1\right)$ are both contained in the set, whereas $\frac{1}{2}y_1 + \frac{1}{2}y_2 = \left(1,\sqrt{\frac{2}{3}},0\right)$ is not. See also Figure \ref{fig:nonConvExt}.
	\end{enumerate}
	
\end{rem}



	\begin{rem} \label{rem:angleMetric} It is not hard to convince oneself that 
		\begin{align*}
		\delta(x,y) = \arccos(\sprod{x,y})
		\end{align*}
		defines a metric on $\sph^{d-1}$. In particular, the triangle inequality holds:
		\begin{align*}
		\arccos(\sprod{x,y})\leq\arccos(\sprod{x,z})+\arccos(\sprod{z,y}).
		\end{align*}
	\end{rem}

After these preparations, we may state our robustness condition.
\begin{prop} \label{lem:geomConvStab}
	Let $x_0 \in \calH$ and $ A : \calH \to \R^m$ be linear. Consider the program $(\calP_f^\epsilon)$,
	where $f: \calH \to \R$ is convex and $b= Ax_0 +n$ with $\norm{n}\leq \epsilon$. If there exists a $\theta>0$ such that
	\begin{align} \label{eq:geomConvStab}
		\mathcal{D}^{\wedge \theta}(f,x_0) \cap \ker A = \set{0},
	\end{align}
	then there exists a constant $\kappa>0$ so that the solution $\hat{x}$ of $(\mathcal{P}^\epsilon_f)$ obeys
	\begin{align*}
		\norm{ \hat{x}-x_0} \leq \kappa \epsilon.
	\end{align*}
	$\kappa$ depends on $\mu$ and the smallest non-zero singular value $\sigma_{\min}(A)$ of $A$.
\end{prop}

For simplicity, we will call \eqref{eq:geomConvStab} the \emph{$\theta$-angular separation condition}, or $\theta$-$ASC$. We will in fact not prove this proposition directly. Instead, we first establish a connection to so called \emph{restricted singular values} of the matrix $A$, which then implies Proposition \ref{lem:geomConvStab} as a corollary.

\begin{figure}
\begin{centering}
	\includegraphics[scale=.3]{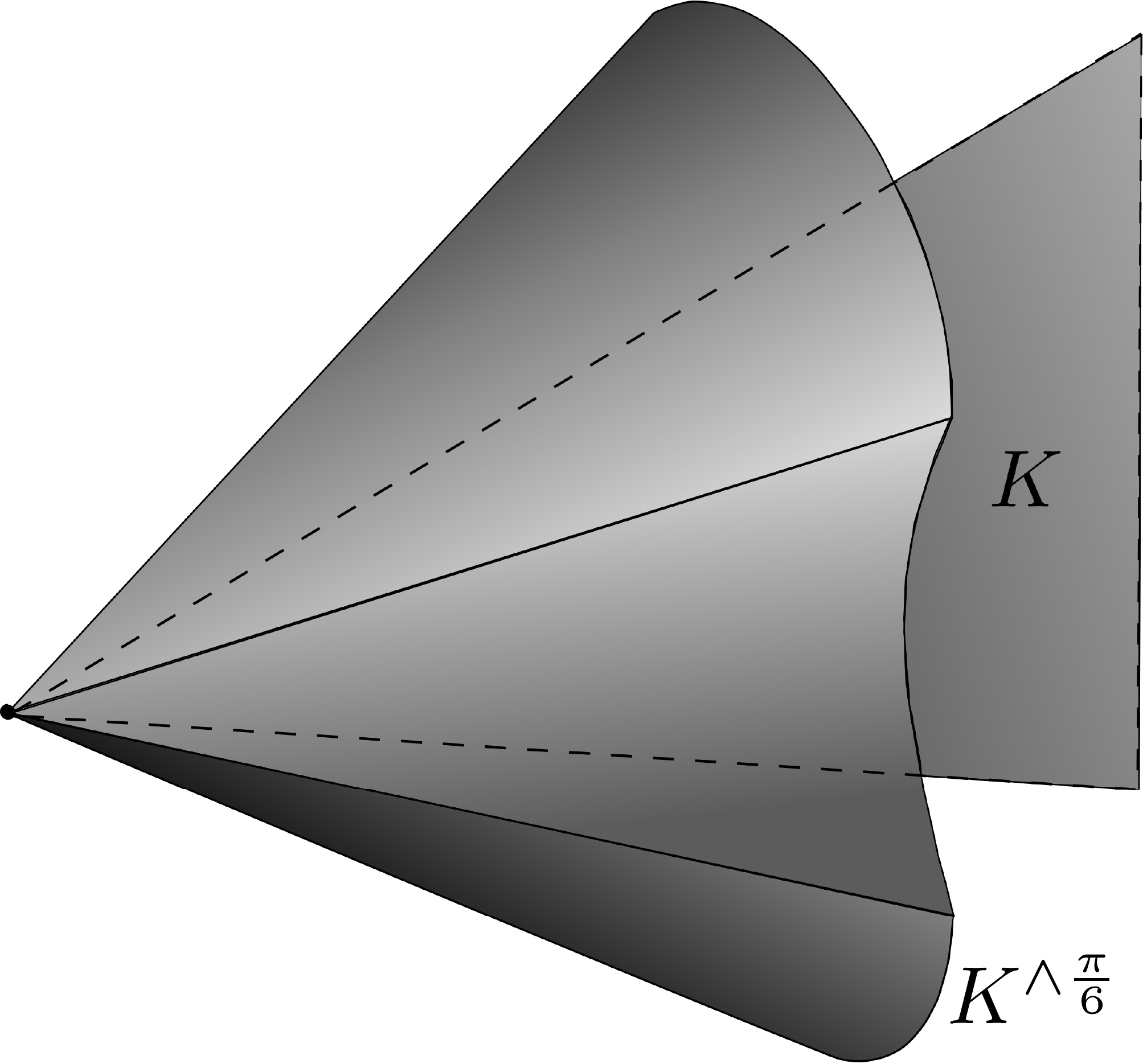}
	\caption{The convex cone $K$ and its non-convex $\frac{\pi}{6}$-extension $K^{\wedge \frac{\pi}{6}}$ \label{fig:nonConvExt}}
	\end{centering}
\end{figure}

The concept of restricted singular values was extensively studied in the paper \cite{amelunxen2014gordon}. There, the singular value of a linear map $A: \calH \to \R^m$ restricted to the cones $C \sse \calH$ and $D \sse \R^m$ was defined as by
\begin{align*}
	\sigma_{C \to D} (A) = \min_{x \in C, \norm{x}=1} \norm{\Pi_D Ax}_2,
\end{align*}
where $\Pi_D$ denotes the \emph{Euclidean} (\emph{metric}, \emph{orthogonal}) \emph{projection} (or \emph{nearest point map}) of the convex set $D$:
	$$ \Pi_D(x) = \argmin_{y \in D} \norm{x-y}_2.$$
 (See also, for instance, \cite{AmelunxLotzMcCoyTropp2014}.) If $D= \R^m$, one also speaks of the \emph{minimal gain} \cite{chandrasekaran2012convex}. 

Interesting for us is that if $\sigma_{\calD(f,x_0) \to \R^m} (A)>0$, $(\calP_f^\epsilon)$ will be robust.
\begin{lem} \label{lem:singValMeansStab} (See \cite{amelunxen2014gordon},
\cite[Proposition 2.2]{chandrasekaran2012convex}.) Let $f: \calH \to \R$ be convex and $A: \calH \to \R^m$ be linear with $m< \dim \calH$. If $\sigma_{\calD(f,x_0) \to \R^m} (A)>0$, any solution $x^*$ of $(\calP_f^\epsilon)$ obeys
\begin{align*}
	\norm{x^* - x_0}_\calH \leq \frac{2\epsilon}{\sigma_{\calD(f,x_0) \to \R^m}}.
\end{align*}
\end{lem}

 \begin{rem} \begin{enumerate}[(i)]
\item Note that in the references we cited, the lemma is proved under slightly less general conditions (e.g., the function $f$ is assumed to be a norm). The proof does however work, line for line, also in our case, and we therefore omit it, for the sake of brevity.
 
\item 	As is indicated by the formulation of  Lemma \ref{lem:singValMeansStab}, the solution vector $x^*$ by no means has to be unique, even for arbitrarily small $\epsilon >0$. We can construct an example of such a situation as follows: Consider a sparse vector $x_0 \in \R^d$ supported on the set $S$ and a matrix $A \in \R^{m,d}$ such that $\sigma_{\calD(\norm{\cdot}_1, x_0) \to \R^m}>0$ - i.e. in particular that $x_0$ is recovered from its exact measurements by $\calP_{1}$. Suppose that the solution $x^*$ of $\calP_{1}^\epsilon$ is unique for some $\epsilon >0$ and is supported on a set $S'$ which is larger than $S$ (this is arguably the most common situation). For any $i \in S' \backslash S$, consider the matrix $\widetilde{A} \in \R^{m,d+1}$ formed by concatenating $A$ with a copy of the $i$:th column of $A$, and the vectors $\tilde{x}^*$ and $\tilde{x}_0$ formed by concatenating $x^*$ and $x_0$ with a zero, respectively. It is then clear that $\tilde{x}_0$ still is recovered from the exact measurements $b= \widetilde{A}\tilde{x}_0$ via the exact program $\calP_{1}$, and hence (see Section \ref{sec:ExactImpliesStable}) $\sigma_{\calD(\norm{\cdot}, \tilde{x}_0) \to \R^m}( \tilde{A})>0$, but that any vector
\begin{align*}
	 \tilde{x}^*_\theta = \theta \tilde{x}^* + (1-\theta) e_{d+1}
\end{align*}
for $\theta \in [0,1]$ solves $\calP_{1}^\epsilon$. Hence, the solution of the relaxed problem $\calP_{1}^\epsilon$ for $\widetilde{A}$ is not unique.
\end{enumerate}
 \end{rem}
 
 We now prove that the $ASC$ to some extent is equivalent to $\sigma_{\calD(f,x_0) \to \R^m} (A)>0$.
 
\begin{lem} \label{lem:singValAndGeomConv} Let $C \sse \calH$ be a non-empty convex cone and $A : \calH \to \R^d$ be a linear map. Then the following are equivalent
\begin{enumerate}[(1)]
	\item There exists a $\theta>0$ such that $C^{\wedge \theta} \cap \ker A = \set{0}$.
	\item $\sigma_{C \to \R^m}(A) >0$.
\end{enumerate}
In particular, if $\sigma_{\min/\max}(A)$ denotes the smallest/largest non-vanishing singular value of $A$, respectively, we have for every $\theta$ with $C^{\wedge \theta} \cap \ker A = \set{0}$ that
\begin{align} \label{eq:singValAndGeomConv}
	\sin(\theta) \sigma_{\min}(A) \leq \sigma_{C \to \R^m}(A) \leq \sin\left( \theta \right) \sigma_{\max}(A).
\end{align}
\end{lem}

\begin{proof}
	$(1) \Rightarrow (2)$. Suppose that $C^{\wedge \theta} \cap \ker A = \set{0}$ and let $x \in C$ have unit norm. Then we have for every $y \in \ker A$
	\begin{align*}
		\norm{x-y}_2^2 = 1 + \norm{y}^2 - 2 \sprod{x,y} \geq 1+ \norm{y}^2-2\cos(\theta)\norm{y} = 1 - \cos^2(\theta) + ( \norm{y}-\cos(\theta))^2 \geq \sin^2(\theta),
	\end{align*}
	since due to $C^{\wedge \theta} \cap \ker A = \set{0}$, there must be $\sprod{x,y} \leq \cos(\theta)\norm{y}$. Since $y \in \ker A$ was arbitrary, we obtain $\norm{x + \ker A} \geq \sin(\theta)$. This has the consequence
	\begin{align*}
		\norm{Ax} \geq \sigma_{\min}(A) \norm{x + \ker A} \geq \sigma_{\min}(A)\sin(\theta).
	\end{align*}
	It follows that $\sigma_{C \to \R^m}(A) \geq \sigma_{\min}(A)\sin(\theta)>0$, since $\sigma_{\min}(A)>0$ due to $\ker A \neq \calH$, which follows from $ C^{\wedge \theta} \cap \ker A = \set{0}$, and $\theta>0$.
	
	$(2) \Rightarrow (1)$ Suppose that $\sigma_{C \to \R^m}(A)>0$ and let $x \in C$ and $y \in \ker A$ have unit norm. Define $\theta$ through $\sprod{x,y} = \cos(\theta)$. Our goal is to prove that $\theta$ has to be larger than some number $\theta_0>0$. 
	
	Since $x \in C$, we have $\norm{Ax}_2 \geq \sigma_{C \to \R^m}(A)$. We also have for every $t \in \R$, due to $y \in \ker A$,
	\begin{align*}
		\sigma_{C \to \R^m}(A)\leq\norm{Ax}_2 = \norm{A(x-ty)}_2 \leq \sigma_{\max}(A) \norm{x-ty}_2 = \sigma_{\max}(A) \sqrt{1+t^2 -2t\cos(\theta)}.
	\end{align*}
	Choosing $t =\cos(\theta)$ yields $$\sin\left(\theta \right)=\sqrt{1-\cos^2(\theta)}\geq \frac{\sigma_{C \to \R^m}(A)}{\sigma_{\max}(A)}>0 ,$$ where we in the last step used $\sigma_{C \to \R^m}(A)>0$. This proves the claim.
\end{proof}

Now Proposition \ref{lem:geomConvStab} easily follows from combining Lemma \ref{lem:singValMeansStab} with Lemma \ref{lem:singValAndGeomConv}. We will return to the connection between restricted singular values and the $ASC$ in the next section.

\begin{rem} \label{rem:stableL2NoAngle}

	 The $ASC$ is not necessary for robust recovery. Consider for example $\ell_2$-minimization in $\calH = \R^d$: 
	\begin{align} \tag{$\mathcal{P}_2^\epsilon$}
		\min \norm{x}_2 \text{ subject to } \norm{Ax -b}\leq \epsilon.
	\end{align}
	In order for $\ell_2$-minimization to exactly recover a signal $x_0$ (which of course is necessary for robust recovery, choose $\epsilon=0)$, we need to have $x_0 \perp \ker A$, since the solution of $(\calP_2)$ is given by $\Pi_{\ker A^\perp} x_0$. Furthermore, $\mathcal{D}(\norm{\cdot}_2, x_0)$ is given by $\set{ v \in \R^d \ \vert \ \sprod{x_0,v} < 0}$. We claim that this implies that for each $\theta>0$, $\mathcal{D}^{\wedge \theta}(\norm{\cdot}_2, x_0) \cap \ker A \neq \set{0}$. 
	
	To see why, let $v \in \mathcal{D}(\norm{\cdot}_2,x_0)$ and $ \eta$ be a nonzero element of $\ker A$. Such an element necessarily exists as soon as $d>m$. For any $\lambda >0$, $v + \lambda \eta \in \mathcal{D}(\norm{\cdot}_2, x_0)$. This since we  have $\sprod{x_0,\eta}=0$ due to $x_0 \perp  \ker A \ni \eta$.  Consequently, $\sprod{v + \lambda \eta, x_0} = \sprod{v,x_0}<0$, i.e. $v + \lambda \eta \in \mathcal{D}(\norm{\cdot}_2, x_0)$. 
	
	Now consider the quotient
	\begin{align*}
		\frac{\sprod{v+ \lambda \eta, \eta} }{\norm{v+\lambda \eta}_2\norm{\eta}_2} \geq \frac{\sprod{v, \eta}+ \lambda\norm{ \eta}_2^2 }{(\norm{v}_2+\lambda \norm{\eta}_2)\norm{\eta}_2}.
\end{align*}	
By letting $\lambda \to \infty$, this quotient can be made arbitrarily close to $1$. Since $v + \lambda \eta \in \mathcal{D}(\norm{\cdot}_2, x_0)$, this means that $\ker A \ni \eta~\in~\mathcal{D}^{\wedge \theta}(\norm{\cdot}_2, x_0)$ for every $\theta >0$. Hence, we have a non-trivial intersection  between the $\theta$-expansion of the descent cone and the kernel for every $\theta > 0$.

	 It is, however, not hard to convince oneself that the solution $\hat{x}$ of $(\calP_2^\epsilon)$ necessarily lies in $\ker A^\perp$ (any part in $\ker A$ can be removed without affecting $\norm{Ax -b}_2$, and at the same time making $\norm{x}_2$ smaller). We already argued that $x_0$ also has this property. This has the immediate consequence that
	\begin{align*}
		\norm{\hat{x}-x_0}_2 \leq \sigma_{\min}(A)^{-1} \norm{A\hat{x} - A x_0}_2 \leq 2 \sigma_{\min}(A)^{-1} \epsilon,
	\end{align*}
	i.e. we have robustness.
	
\end{rem}

Now we prove that under the assumption that $f$ is a norm on $\calH$, the $ASC$ in fact also implies stability. 

\begin{theo} \label{prop:robustCrit}
	Let $\norm{\cdot}_*$ be a norm on $\calH$ and consider the convex program
	\begin{align} \label{eq:normProg}
		\min \norm{x}_* \text{ subject to } \norm{Ax-b} \leq \epsilon,
	\end{align}
	where $b=A\check{x}_0 +n$ with $\norm{n} \leq \epsilon$. Suppose that there exists a $\theta>0$ so that $A$ fulfills the $\theta$-$ASC$ for every $x_0$ in some subset $\calC \sse \calH$.
	Then the following is true for \eqref{eq:normProg}: There exist constants $\kappa_1$ and $\kappa_2$ such that any solution $x^*$ fulfills
	\begin{align} \label{eq:robustEq}
		\norm{x^*- \check{x}_0} \leq \kappa_1 \epsilon + \kappa_2 \dist_{\norm{\cdot}_*}(\check{x}_0, \cone(\calC)) .
	\end{align}
	Here, $\cone(\calC)$ denotes the cone generated by $\calC$, i.e., the set $\set{\lambda x_0, \lambda >0 \text{ and } x_0 \in \calC}$. The first constant $\kappa_1$ depends on  $\sigma_{\min}(A)$,  $\norm{\cdot}_*$ and $\theta$. The second constant $\kappa_2$ depends on $\theta$, $\norm{\cdot}_*$ and the condition number $\sigma_{\max}(A)/\sigma_{\min}(A)$ of $A$.
\end{theo}

\begin{proof} (See Fig. \ref{fig:rob} for a graphical depiction of the proof.)
Let $x_0$ be a, not necessarily unique, vector in $\cone(\calC)$ with $\norm{\check{x}_0-x_0}_{*} = \delta:= \dist_{\norm{\cdot}_{*}}(\check{x}_0, \cone(\calC))$. Due to the homogenity of $\norm{\cdot}_{*}$, we have $\calD(\norm{\cdot}_*,x_0) = \calD(\norm{\cdot}_*,  \lambda x_0)$ for every $\lambda >0$. Therefore we may without loss of generality scale the problem such that $\norm{x_0}_*=1$. Also, since all norms on the finite dimensional space $\calH$ are equivalent, there exists $\gamma >0$ so that for each $x\in \calH$, $\gamma^{-1} \norm{x}_{*} \leq \norm{x} \leq \gamma \norm{x}_{*}$. These two facts have the consequence that
\begin{align*}
	\norm{A\check{x}_0 - Ax_0} \leq \sigma_{\max}(A) \norm{\check{x}_0-x_0}\leq \gamma \sigma_{\max}(A) \delta.
\end{align*}
Since $\norm{Ax^*-A\check{x}_0} \leq \norm{Ax^*-b} + \norm{b-A\check{x}_0} \leq 2 \epsilon$, this implies that
\begin{align*}
	\norm{Ax^* - Ax_0}\leq \norm{Ax^* - A \check{x}_0} + \norm{A \check{x}_0 - Ax_0} \leq 2\epsilon + \gamma \sigma_{\max}(A) \delta,
\end{align*}
i.e. $\norm{x_0 - \hat{x} + \ker A} \leq \sigma_{\min}(A)^{-1}(2\epsilon + \gamma \sigma_{\max}(A) \delta)$. Let $h$ be the vector in $\ker A$ so that $\norm{x_0 + h - \hat{x}} = \norm{x_0 - \hat{x} + \ker A}$.

\begin{figure}
\centering
\includegraphics[scale=.3]{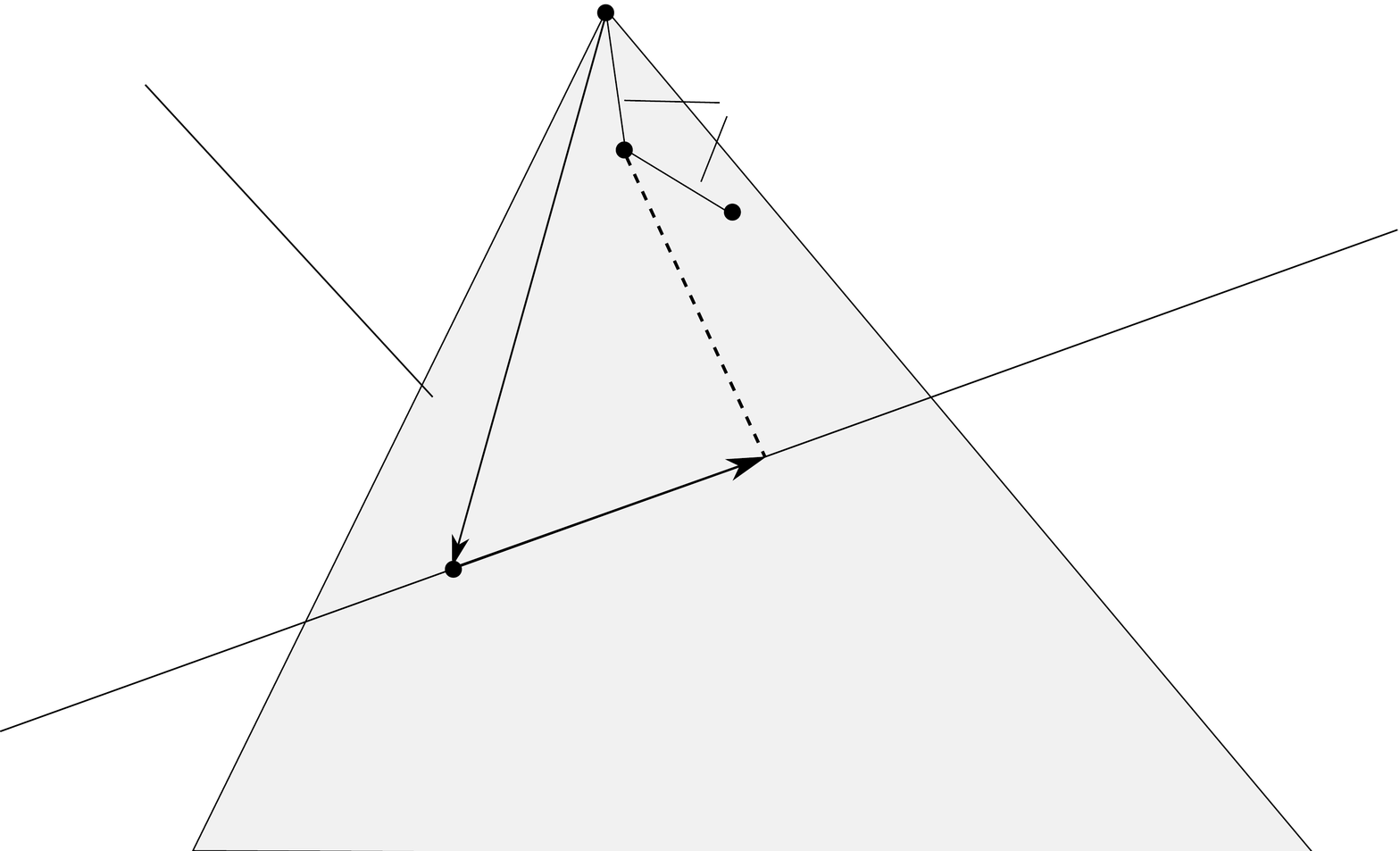}
	\caption{The proof of Proposition \ref{prop:robustCrit}. \label{fig:rob}}
\end{figure}

Now since $x^*$ is a solution of \eqref{eq:normProg}, there must be $\norm{x^*}_*\leq \norm{\check{x}_0}_* \leq \norm{x_0}_* +\delta= \norm{(1+\delta)x_0}_*$, i.e. $\tilde{h} := x^* - (1+\delta)x_0 \in \mathcal{D}(\norm{\cdot}_*,(1+\delta) x_0)=\mathcal{D}(\norm{\cdot}_*,x_0)$, where the latter is due to the homogeneity of $\norm{\cdot}_*$.

Now we have
\begin{align*}
	\norm{x^*-x_0} \leq \norm{x^* - (1+\delta)x_0} + \norm{\delta x_0} \leq \Vert \tilde{h}\Vert + \gamma \delta.
\end{align*}
Due to $h \in \ker A$ and $\tilde{h} \in \mathcal{D}(\norm{\cdot},x_0)$, \eqref{eq:geomConvStab} implies that $\langle h, \tilde{h} \rangle \leq \cos(\theta)  \norm{h}\Vert{\tilde{h}}\Vert$. This has the consequence that
\begin{align*}
	\Vert h - \tilde{h} \Vert^2 \geq \norm{h}^2 + \Vert \tilde{h} \Vert^2 - 2 \cos(\theta) \norm{h} \Vert \tilde{h} \Vert = ( \norm{h} - \cos(\theta) \Vert \tilde{h} \Vert )^2 + (1 - \cos^2(\theta)) \Vert \tilde{h} \Vert^2 \geq \sin^2(\theta) \Vert \tilde{h} \Vert^2,
\end{align*}
 which implies that $\Vert \tilde{h}\Vert_2 \leq \sin^{-1}(\theta)\Vert h - \tilde{h}\Vert$. Since
\begin{align*}
	\Vert h-\tilde{h} \Vert = \norm{h- x^* + (1+\delta)x_0} \leq \norm{x_0 + h - x^*} + \delta \norm{x_0} \leq \sigma_{\min}(A)^{-1}(2\epsilon + \gamma \sigma_{\max}(A) \delta) + \gamma \delta,
\end{align*}
we have $\Vert \tilde{h}\Vert \leq \left(\sin(\theta)\sigma_{\min}(A)\right)^{-1}( 2\epsilon + \gamma \sigma_{\max}(A) \delta) + \sin^{-1}(\theta)\gamma \delta $. Finally, we estimate
\begin{align*}
	\norm{x^*-\check{x}_0}_* &\leq \norm{x^* - x_0}_* + \norm{x_0 - \check{x}_0}_* \leq \gamma \norm{x^*-x_0} + \delta \leq \gamma(\Vert \tilde{h}\Vert + \gamma \delta) + \delta \\
	&\leq \sin^{-1}(\theta) \left(\gamma\sigma_{\min}(A)^{-1}( 2\epsilon + \gamma \sigma_{\max}(A) \delta) +  \delta \right) + 2 \delta =: \kappa_1 \epsilon + \kappa_2 \dist_{\norm{\cdot}}(\check{x}_0, \cone (\calC) ),
\end{align*}
where $\kappa_1 =2\gamma (\sin(\theta)\sigma_{\min}(A))^{-1}$ and $\kappa_2 = \sin\theta^{-1}(\gamma^2\sigma_{\max}(A)/\sigma_{\min}(A) +1)+2$, which is what we wanted to prove.
\end{proof}

\subsection{$ASC$ and Two Other Geometrical Notions of Stability.} \label{sec:ASCvsConds}
Let us end this section by briefly discussing the connection between the $ASC$-condition and two other measures for stability of a linear embedding, namely the so-called \emph{Renegar's condition number}\cite{BelloniFreund2007,renegar1992} $\calC_\calR(A)$ and the \emph{Grassmannian condition number} \cite{amelunxenBurgisser2012}  $\calC(A)$ of a matrix. They were originally introduced to study the stability of the \emph{homogeneous convex feasibility problem}: Given a closed convex cone $K \sse \R^m$ with non-empty interior not containing a subspace (a \emph{regular} cone), for which  $A \in \R^{m,d}$ does there exist a $z$ with
\begin{align*}
	Az \in \text{int } K ?
\end{align*}
Let us call matrices such \emph{feasible}. The connection to our problem follows from duality: given a cone $C$ whose polar $C^*= \set{ x \in \R^d \ \vert \ \forall y \in C : \sprod{x,y}\leq 0 } $ is regular,  it is well known that if the range of the transposed matrix $A^*$ intersects the $\text{int } C^*$ , the kernel of a matrix $A$ can't intersect $C$ non-trivially (see for instance \cite{BelloniFreund2007}).

 Since the range of $A^*$ always is equal to the orthogonal complement $\ker A^\perp$ of $\ker A$, it makes sense to define the following sets of subspaces :
\begin{align*}
	P_m(K) = \set{U \in \gr(d,m) \ \vert \ U^\perp \cap K^* \neq \set{0}}, D_m(C) = \set{U \in \gr(d,m) \ \vert \ U \cap K \neq \set{0}}
\end{align*}
$\gr(d,m)$ denotes the \emph{Grassmannian manifold} of $m$-dimensional subspaces of $\R^d$. It can be proven that both $P_m(K)$ and $D_m(K)$ are closed in $\gr(d,m)$ and that they share a common boundary $\Sigma_m(K)$. The Grassmannian condition number of a matrix $A$ with respect to a regular cone $C$ is defined as the inverse of the distance of $\ker A$ to $\Sigma_m(C)$, i.e.
\begin{align*}
	\calC(A) =  \frac{1}{\dist(\ker A, \Sigma_m(C))}.
\end{align*}
The distance is thereby calculated with respect to the canonical metric on $\gr(d,m)$: if $U$ and $V$ are $m$-dimensional subspaces of $\R^d$ and $\Pi_U$ and $\Pi_V$ are the orthogonal projections onto them, we define
\begin{align*}
	d(U,V)= \norm{\Pi_U - \Pi_V}.
\end{align*}
Here, $\norm{\cdot}$ denotes the operator norm.

$\calC(A)$ is very closely related to the $ASC$-condition: if $A$ is feasible, the largest angle $\theta^*$ so that $A$ satisfies the $\theta$-$ASC$ satisfies \cite[Proposition 1.6]{amelunxenBurgisser2012} 
\begin{align} \label{eq:grassmannCond}
	\sin(\theta^*) = \frac{1}{\calC(A)}.
\end{align}

The Grassmannian condition number is itself closely related to the so-called Renegar's condition number $\calC_\calR(A)$. Given a feasible matrix $A$ (in the same sense as above), it is defined as the distance from $A$ to the set of infeasible matrices, i.e.
\begin{align*}
	\calC_\calR(A) = \min \set{ \norm{\Delta A} \ \vert \ A + \Delta A \text{ is infeasible} }.
\end{align*}
In our setting, it can be proven that \cite[Lemma 2.2]{Roulet2015Renegar}
\begin{align} \label{eq:renegar}
	\calC_\calR(A) = \frac{\sigma_{\max}(A)}{\sigma_{C \to \R^m}(A)}.
\end{align}

The article \cite{Roulet2015Renegar} contains some more details and interesting results concerning the connection between $\calC_\calR(A)$ and the robustness properties of compressed sensing problems. In particular, they prove the following version of Lemma \ref{lem:singValMeansStab} of this article: if we assume that the noise level is below $\norm{A} \epsilon$ (which in particular is interesting when the measurement matrix $A$ only is known up to some error $\Delta A$), every solution $x^*$ of the program $\calP_f^\epsilon$ obeys
\begin{align*}
	\norm{x^*-x_0}_2 \leq 2 \calC_\calR(A) \epsilon.
\end{align*}

Let us end this section by noting that one can prove Lemma \ref{lem:singValAndGeomConv} by using the following inequality from \cite[Theorem 1.4]{amelunxenBurgisser2012}   :
\begin{align*}
	\calC(A) \leq \calC_\calR(A) \leq \frac{\sigma_{\max}(A)}{\sigma_{\min}(A)}\calC(A).
\end{align*}
Using \eqref{eq:grassmannCond} and \eqref{eq:renegar}, this can be rewritten as
\begin{align*}
	\sin(\theta^*)\sigma_{\min}(A)  \leq \sigma_{C \to \R^m}(A) \leq \sin(\theta^*) \sigma_{\max}(A),
\end{align*}
which is exactly the inequality \eqref{eq:singValAndGeomConv}.

\section{When is the $ASC$ satisfied?}

Having established that the $ASC$ implies stability and robustness for signal recovery using the convex program \eqref{eq:normProg}, it is of course interesting to ask for which matrices this condition is satisfied. In this section, we will first prove the maybe somewhat remarkable fact that, for many reasonable norms, the weak $NSP$-like condition \eqref{eq:geomExactCond} in fact \emph{implies} that the $ASC$ is satisfied for some $\theta>0$.

However, the above reasoning only yields the \emph{existence} of a $\theta>0$ with \eqref{eq:geomConvStab}, and does not give any control of the size of $\theta$. Therefore, we will also briefly discuss the relation between already known stability conditions for compressed sensing, and that the $ASC$ can be secured with high probability using random Gaussian matrices. Using the concept of $\emph{Gaussian widths}$, we will argue that if one needs $m_0$ measurements to secure that $(\calP_f)$ recovers a signal $x_0$ with high probability from noiseless measurements, we need $m_0 + O\left(\sin\left(\frac{\theta}{2}\right)d\right)$ to secure the $ASC$ for $\theta>0$.

\subsection{Exact Recovery Implies Some Stability and Some Robustness.} \label{sec:ExactImpliesStable}

The first result of this subsection was essentially already proven in \cite{amelunxen2014gordon}. Let us state it, and for completeness also give a proof, and then discuss its implications and limitations.

\begin{theo} Let $C \sse \calH$ be a \emph{closed} convex cone, and $A: \calH \to \R^m$ be linear. Then if $ C \cap \ker A= \set{0}$, there exists a $\theta>0$ such that $C^{\wedge \theta} \cap \ker A = \set{0}$, i.e., the $\theta$-$ASC$ holds.
\end{theo}

\begin{proof}
Under the assumption that $C$ and $D$ are closed, the restricted singular value $\sigma_{C \to D}(A)$ vanishes if and only if either $A C \cap D^* \neq \set{0}$ or $C \cap \ker A \neq \set{0}$ \cite[Proposition 2.2]{amelunxen2014gordon}. Since $D=\R^m$ in our case, $D^* = \set{0}$, and we hence by contraposition have the equivalence
\begin{align*}
	\sigma_{C \to \R^m}(A) >0 \Leftrightarrow C \cap \ker A = 0.
\end{align*} 
Since by Lemma \ref{lem:singValAndGeomConv}, $\sigma_{C \to \R^m}(A) >0$ is equivalent to the existence of a $\theta>0$  such that $C^{\wedge \theta} \cap \ker A = \set{0}$, the claim is proven.
\end{proof}

On a theoretical level, the last proposition implies that as soon as the recovery of some class of signals $\calC$ from exact measurements with the help of a convex program is guaranteed, we also have stability and robustness for the recovery of signals close to $\calC$ from noisy measurements. As simple and beautiful the result is, it has its flaws. In particular, we have no control whatsoever over the size of the parameter $\theta$, which in turn implies that we have no control over the constants in \eqref{eq:robustEq}. 

In the case that the norm $\norm{\cdot}_*$ has a unit ball which is a polytope, we can do a bit better. Although we still cannot provide any general bound on the size of $\theta$, we can prove that it will have the same size for all points lying in the same face of the unit ball. 

Before stating and proving the result, let us note that the assumption that the unit ball of $\norm{\cdot}_*$ is a polytope is not far-fetched. In particular, it is true for both $\ell_1$-minimization (and its many variants, i.e., also for weighted norms etc.), and for $\ell_\infty$-minimization -- or in general any atomic norm generated by a finite set of atoms $\calA$. Let us now formulate the main part of the argument in the following lemma.

\begin{lem} \label{lem:boundingAngle}
	 Let $P \sse \calH$ be a closed polytope and $\calC \sse P$ be a union of faces of $P$. Suppose that the linear subspace $U \sse \calH$ has the property that for each $x_0 \in \cal C$, $x_0 + U$ intersects $P$ only in $x_0$. Then for each $x_0 \in \calC$,  there exists a $\mu<1$ such that
	 \begin{align*}
	 	\forall x \in P, z \in U : \sprod{x -x_0 , z} \leq \mu \norm{x-x_0}_2\norm{z}_2.
	 \end{align*}
	The size of $\mu$ is only dependent on which face $x_0$ lies in.	
\end{lem}

Although the proof of this lemma is elementary, it is relatively long. Therefore, we postpone it to Appendix \ref{app:A}. Instead, we use it to prove the aforementioned result about stability and robustness for recovery using convex programs involving norms with polytope unit balls.

\begin{cor}
Let $A: \calH \to \R^m$ be given. Suppose that $\norm{\cdot}_{*}$ is a norm whose unit ball is a polytope, and $\calC$ be a union of faces of that polytope. If the program $(\calP_{\norm{\cdot}_{*}})$ recovers $x_0$ from the noiseless measurements $Ax_0$ for every $x_0 \in \calC$,  all signals $\check{x}_0$ close to the cone generated by $\calC$ will be stably and robustly recovered by $(\calP_{\norm{\cdot}_{\ast}}^{\epsilon})$ in the sense of \eqref{eq:robustEq}. The constants $\kappa_1$ and $\kappa_2$ will only depend on which face the normalized version $\check{x}_0/ \norm{\check{x}_0}_*$ of $\check{x}_0$ lies closest to.
\end{cor}

\begin{proof}
	Since each $x_0$ in $\calC$ is recovered exactly by $(\calP_{\norm{\cdot}_*})$ by noiseless measurements, we will by Lemma \ref{lem:descConeKernel} have $\calD( \norm{\cdot}_*,x_0)\cap \ker A = \set{0}$ for each $x_0 \in \cal C$. Since the descent cone of $\norm{\cdot}_*$ at $x_0$ is generated by the vectors $x -x_0$, where $x \in P = \set{y \vert \norm{y}_*\leq 1}$, the conditions of Lemma \ref{lem:boundingAngle} are satisfied. Said lemma therefore implies that $\calD^{\wedge \theta}( \norm{\cdot}_*,x_0)\cap \ker A = \set{0}$ for $x_0 \in \calC$, where $\theta>0$ only depends on which face $x_0$ lies in. This together with Theorem \ref{prop:robustCrit} implies the claim.
\end{proof}




\subsection{$ASC$ Compared to Classical Stability and Robustness Conditions in Compressed Sensing.}
In the following, we will relate the $ASC$ to two well-known criteria for stability and robustness of $\ell_1$-minimization from the literature: the $RIP$ and the $RNSP$. We will begin by considering the $RIP$. It is a well-known fact that if the restricted isometry constant $\delta_{2s}$ is small, the program $(\calP_{\norm{\cdot}_1}^\epsilon)$ will recover any $s$-sparse vector in a robust and stable manner. E.g. in \cite[Theorem 6.12]{MathIntroToCS}, it is proved that if $\delta_{2s} < 4/ \sqrt{41}$, \eqref{eq:robustEq} will be satisfied for some constants $\kappa_1$, $\kappa_2$ only dependent on $\delta_{2s}$. Having this in mind, it is of course interesting to ask oneself if it is possible to directly prove that a small $\delta_{2s}$ will imply the $ASC$ for some $\theta>0$. The next proposition gives a positive answer to that question, and it furthermore provides the control of the size of $\theta>0$ we lacked in the previous section.

\begin{prop} \label{prop:RIPASC}
	Suppose that $\delta_s$ and $\delta_{2s}$ of the matrix $A$ satisfies
	\begin{align*}
		\frac{1}{1-\delta_s} \left(\delta_s + \sqrt{5}\sqrt{\frac{d}{s}+1} \frac{\delta_{2s}}{1+\delta_s} \right) \leq \cos(\theta).
	\end{align*}
	Then $\calD^{\wedge \theta}(\norm{\cdot}_1, x_0) \cap \ker A = \set{0}$ for every $s$-sparse $x_0$.
\end{prop}

For the proof of this claim, we need two lemmata. The first one, we cite from the book \cite{MathIntroToCS}.

\begin{lem} \label{lem:RIPROC} (See \cite[Proposition 6.3]{MathIntroToCS}). Let $u$ and $v$ be $s$-sparse vectors with disjoint support. Then we have
\begin{align*}
	\abs{\sprod{Au,Av}} \leq \delta_{2s}\norm{u}_2 \norm{v}_2,
\end{align*}
where $\delta_{2s}$ is the $(2s)$-th restricted isometry constant of $A$.
\end{lem}

The next one is about the structure of vectors in the descent cone of the $\ell_1$-norm at a sparse vector.

\begin{lem} \label{lem:normIneqL1Desc} Let $x_0 \in \R^d$ be supported on the set $S_0$ with $\abs{S_0}=s$, and $u \in \mathcal{D}(\norm{\cdot}_1, x_0)$. Let furthermore $(S_i)_{i=1}^n$ be a partition of the set $\set{1, \dots, d} \backslash S_0$ with the properties
\begin{align} \label{eq:monCrit}
	\min_{i \in S_k} \abs{u(i)} & \geq \max_{i \in S_{k+1}} \abs{u(i)} \text{  and  }
	 \abs{S_k}  \leq s 
\end{align} 
for every $k=1, \dots n-1$. Then we have
\begin{align*}
	\sum_{i=0}^n \norm{u_{S_i}}_2 < \sqrt{5}\norm{u}_2.
\end{align*}
\end{lem}
\begin{proof}
	First, \eqref{eq:monCrit} implies that $\norm{u_{S_{k+1}}}_2 \leq \frac{1}{\sqrt{s}} \norm{u_{S_k}}_1$ for $k=1 , \dots n$ \cite[Lemma 6.10]{MathIntroToCS}. Therefore, we have the following estimate due to $\norm{v}_2 \leq \norm{v}_1$ for $v \in \R^d$:
	\begin{align*}
		\sum_{k=0}^n \norm{u_{S_k}}_2 \leq \norm{u_{S_0}}_2+ \norm{u_{S_1}}_2 +\frac{1}{\sqrt{s}}\sum_{k=2}^n \norm{u_{S_{k-1}}}_1 \leq \norm{u_{S_0}}_2 + \Vert u_{S_0^c} \Vert_2 + \frac{1}{\sqrt{s}}\Vert u_{S_0^c} \Vert_1
	\end{align*}
	Now, since $u \in \calD(\norm{\cdot}_1, x_0)$, we have (see Example \ref{ex:l1desc})
	\begin{align*}
	\Vert u_{S_0^c}\Vert_1  \leq -\sum_{i \in S_0} \sgn(x_0(i)) u_i \leq \norm{u_{S_0}}_1 \leq \sqrt{s}\norm{u_{S_0}}_2,
	\end{align*}
	where the last inequality is due to the fact that $u_{S_0}$ is supported on $S_0$.  This implies 
	\begin{align*}
		\norm{u_{S_0}}_2 + \Vert u_{S_0^c} \Vert_2 + \frac{1}{\sqrt{s}}\Vert u_{S_0^c} \Vert_1 &\leq 2 \norm{u_{S_0}}_2 + \Vert u_{S_0^c} \Vert_2 = (2,1) \cdot ( \norm{u_{S_0}}_2 , \Vert u_{S_0^c} \Vert_2 ) \\
		&\leq \sqrt{5} \sqrt{ \Vert u_{S_0} \Vert_2^2 + \Vert u_{S_0^c} \Vert_2^2}= \sqrt{5} \norm{u}_2,
	\end{align*}
	where we used the Cauchy-Schwarz inequality in the third step.
\end{proof}

With these two lemmas, we may prove Proposition \ref{prop:RIPASC}.

\begin{proof}[Proof of Proposition \ref{prop:RIPASC}]
	Let $u \in \calD(\norm{\cdot}_1, x_0)$ and $v \in \ker A$, where without loss of generality $\norm{u}_2=\norm{v}_2=1$. Our goal is to prove that necessarily $\sprod{u,v} < \cos(\theta)$. Let us begin by partitioning the set $\set{1, \dots, d} \backslash S_0$ into $n$ sets $S_1, \dots S_n$ so that $S_i \cap S_j = \emptyset$ for $i \neq j$, $\abs{S_i}= s$  for $i=1, \dots n-1$, $\abs{S_n}\leq s$, and the monotonicity criterion \eqref{eq:monCrit} is met for $u$. Then $n \leq \frac{d}{s}$, and  we have
	\begin{align*}
		\sprod{u,v} = \sum_{i=0}^n \sprod{u_{S_i}, v_{S_i}} &= \frac{1}{4} \left(\sum_{i=0}^n \norm{u_{S_i} + v_{S_i}}_2^2 - \norm{u_{S_i} - v_{S_i} }_2^2\right) \\
		&\leq \frac{1}{4} \left( \sum_{i=0}^n\frac{1}{1- \delta_s}\norm{Au_{S_i} + Av_{S_i}}_2^2 -\frac{1}{1+ \delta_s}\norm{Au_{S_i} - Av_{S_i}}_2^2 \right) \\
		&= \frac{1}{4} \left( \sum_{i=0}^n\frac{2\delta_s}{1- \delta_s^2}(\norm{Au_{S_i}}_2^2 + \norm{Av_{S_i}}_2^2) + \frac{4}{1-\delta_s^2}\sprod{Au_{S_i},Av_{S_i}}\right) ,
	\end{align*}
	where we in the second to last step used that $u_{S_i} \pm v_{S_i}$ for each $i$ is $s$-sparse, and the definition of $\delta_s$. Now since $v \in \ker A$, we have $Av_{S_i} = - \sum_{\substack{k=1, k \neq i}}^n Av_{S_k}$, and consequently,
	\begin{align*}
		\sprod{Au_{S_i},Av_{S_i}} = -\sum_{\substack{k=1 \\ k \neq i}}^n \sprod{Au_{S_i},Av_{S_k}} \leq \sum_{\substack{k=1 \\ k \neq i}}^n \delta_{2s} \norm{u_{S_i}}_2 \norm{v_{S_k}}_2 \leq \sum_{\substack{k=1}}^n \delta_{2s} \norm{u_{S_i}}_2 \norm{v_{S_k}}_2,
	\end{align*}
	where we in the second to last step applied Lemma \ref{lem:RIPROC}. Again using the definition of $\delta_s$, we conclude $\norm{Au_i}_2^2 \leq (1+\delta_s)\norm{u_i}_2^2$ and $\norm{Av_i}_2^2 \leq (1+\delta_s)\norm{v_i}_2^2$. Combining all of the previous estimates, we obtain
	\begin{align*}
		\sprod{u,v} \leq  \sum_{i=0}^n\frac{\delta_s}{2(1- \delta_s)}(\norm{u_{S_i}}_2^2 + \norm{v_{S_i}}_2^2) + \frac{\delta_{2s}}{1-\delta_s^2}\left(\sum_{i=0}^n \norm{u_{S_i}}_2\right) \left(\sum_{i=0}^n \norm{v_{S_i}}_2\right)
	\end{align*}
	Now we use $\norm{u},\norm{v}=1$, Lemma \ref{lem:normIneqL1Desc}, the inequality $\sum_{i=0}^n x_i \leq \sqrt{n+1}\sqrt{ \sum_{i=0}^n x_i^2}$ and $n \leq \frac{d}{s}$ to conclude
	\begin{align*}
		 \sprod{u,v} < \frac{\delta_s}{1- \delta_s} +\sqrt{5}\frac{\delta_{2s}}{1-\delta_s^2}\sqrt{\frac{d}{s}+1}\leq \cos(\theta),
	\end{align*}
	which is what we wanted to prove.
\end{proof}

Now let us turn to another criterion for robust and stable recovery using $\ell_1$-minimization: the \emph{Robust Null-Space Property} or $RNSP$. Let $0\leq\gamma<1$ and $\tau>0$. A matrix $A \in \R^{m,d}$ satisfies the $(\gamma,\tau)$-$RNSP$ with respect to the index set $T \sse \set{1,2, \dots, d}$ if for every $x \in \R^d$, we have
	\begin{align}
		\norm{x_T}_1 \leq \gamma \norm{x_{T^c}}_1 + \tau \norm{Ax}_2. \label{eq:RNSP}
	\end{align}
	In fact, it turns out that the $RNSP$ with respect to an index set $T$ is equivalent to that the $ASC$ is fulfilled for any $x_0$ supported on $T$, in the sense specified by the following theorem.
	
	\begin{theo} \label{theo:RNSPvsASC}
		Let $A\in \R^{m,d}$.
		\begin{enumerate}[(1)]
			\item If the $(\gamma,\tau)$-$RNSP$ is satisfied, then the $ASC$-condition is satisfied uniformly for all $x_0$ supported on $T$, i.e. there exists a $\theta>0$ with
			\begin{align}
				\calD^{\wedge \theta}(\norm{\cdot}_1, x_0) \cap \ker A = \emptyset \label{eq:RNSPvsASCEq}
			\end{align}
			for every $x_0$ supported on $T$.
			\item If there exists a $\theta>0$ such that \eqref{eq:RNSPvsASCEq} holds for every $x_0$ supported on $T$, there exists $0\leq\gamma<1$ and $\tau>0$ such that the $(\gamma,\tau)$-$RNSP$ holds with respect to $T$.
		\end{enumerate}
	\end{theo}
	
	\begin{proof}
		$(1)$. Due to Lemma \ref{lem:singValAndGeomConv}, it suffices to prove that $\sigma_{\calD(\norm{\cdot}_1, x_0)}(A) >\sigma_0>0$ for every $x_0$ supported on $T$. To this end, note that the $RNSP$ implies that
		\begin{align*}
			\min_{\substack{x \in \calD(\norm{\cdot}_1, x_0) \\ \norm{x}_2=1}} \norm{Ax}_2 \geq \frac{1}{\tau}\min_{\substack{x \in \calD(\norm{\cdot}_1, x_0) \\ \norm{x}_2=1}} \left(\norm{x_T}_1 - \gamma \Vert x_{T^c} \Vert_1 \right) &\geq \frac{1-\gamma}{\tau}  \min_{\substack{x \in \calD(\norm{\cdot}_1, x_0) \\ \norm{x}_2=1}} \norm{x_T}_1   \\
			&\geq \frac{1-\gamma}{\tau}  \min_{\substack{x \in \calD(\norm{\cdot}_1, x_0) \\ \norm{x}_2=1}} \frac{1}{2}\norm{x}_1 \geq \frac{(1-\gamma)}{2\tau}.
		\end{align*}
		We used that since $x \in \calD(\norm{\cdot}_1,x_0)$, $\norm{x_T}_1 \geq \Vert x_{T^c} \Vert_1$. This in particular implies that $\norm{x_T}_1 \geq \frac{1}{2}\norm{x}_1$ The claim has been proven. 
		
		$(2)$. Suppose that there exists a $\theta$ with \eqref{eq:RNSPvsASCEq} for every $x_0$ supported on $T$. Let us begin by arguing that this implies that there exists a $0\leq\gamma<1$ with the following property: If $\norm{x_T}_1 \geq \gamma \Vert x_{T^c} \Vert_1$, there exists an $x_0$ supported on $T$ with $x \in \calD^{\wedge \frac{\theta}{2}}( \norm{\cdot}_1, x_0)$. 
		 
		 Consider an $x$ with $\norm{x_T}_1 \geq \gamma \Vert x_{T^c} \Vert_1$ (where the value of $\gamma$ is yet to be determined) and consider $x_0 := -x_T$. Then we have  $\tilde{x} \in \calD( \norm{\cdot}_1, x_0)$, where
		\begin{align*}
			\tilde{x}(i) = \begin{cases} x(i) & i \in T \\
			\gamma x(i) & i \in T^c, \end{cases}
		\end{align*}
        since	
                \begin{align*}			
                        \sum_{i \in T^c} \vert \tilde{x}(i)\vert = \gamma\sum_{i \in T^c} \vert {x}(i)\vert \leq \sum_{i \in T} \vert {x}(i)\vert = -\sum_{i \in T} \sgn x_0(i) \tilde{x}(i),			
                \end{align*}                			
	since $\Vert x_T \Vert_1 \geq \gamma\Vert x_{T^c} \Vert_1$ (see also Example \ref{ex:l1desc}). Now we have 
		\begin{align*}
			\frac{\sprod{x, \tilde{x}}}{\norm{x}_2 \norm{\tilde{x}}_2} = \frac{\sum_{i \in T} x(i)^2 + \gamma \sum_{i \in T^c} x(i)^2}{\norm{x}_2 \norm{\tilde{x}}_2} = \frac{\norm{x_T}^2 + \gamma (1-\norm{x_T}^2)}{\sqrt{\norm{x_T}^2 + \gamma^2(1 - \norm{x_T}^2})} \geq \inf_{0 \leq a \leq 1 } \frac{a + \gamma (1-a)}{\sqrt{a +  \gamma^2 (1-a)}}
		\end{align*}
		where we in the second to last step without loss of generality assumed that $\norm{x}_2=1$. Since $0 \leq a, \gamma \leq 1$, we have $\sqrt{a + \gamma^2(1-a)} \leq \sqrt{a+ \gamma(1-a)}$, and consequently $\frac{a + \gamma (1-a)}{\sqrt{a +  \gamma^2 (1-a)}} \geq \sqrt{a+ \gamma(1-a)}$. Hence, we have
 $$\inf_{0 \leq a \leq 1 } \frac{a + \gamma (1-a)}{\sqrt{a +  \gamma^2 (1-a)}} \geq \inf_{0 \leq a \leq 1 }\sqrt{a+ \gamma(1-a)} = \sqrt{\gamma} \geq \cos\left(\frac{\theta}{2}\right),$$
	if we choose $\gamma$ close enough to $1$. This proves that $x \in \calD^{\wedge \frac{\theta}{2}}( \norm{\cdot}_1, x_0)$.
	
	Now we claim that there exist a $\tau>0$ such that the $(\gamma, \tau)$-$RNSP$ with respect to $T$ is satisfied. To see this, first note that \eqref{eq:RNSP} is trivial for $x$ with $\norm{x_T}_1 \leq \gamma \Vert x_{T^c} \Vert_1$. For the other $x$, which we without loss of generality may assumed to be $\ell_2$-normalized, we know by the above argument that $x \in \calD^{\wedge \frac{\theta}{2}}( \norm{\cdot}_1, x_0)$ for some $x_0$ supported on $T$. This implies that $x \notin (\ker A)^{\wedge \frac{\theta}{2}}$, since if indeed $x \in (\ker A)^{\wedge \frac{\theta}{2}}$, there would exist  unit norm vectors $y \in \ker A $ and $z \in \mathcal{D}(\norm{\cdot}_1, x_0)$ such that $\arccos \sprod{x,y} \leq \frac{\theta}{2}$ and $\arccos\sprod{x,z} \leq \frac{\theta}{2}$.  Remark \ref{rem:angleMetric} would then imply that $\arccos\sprod{y, z} \leq \theta$, i.e., $\sprod{x,y} \geq \cos(\theta)$, which is a contradiction to \eqref{eq:RNSPvsASCEq}.
	
	But, $x \notin (\ker A)^{\wedge \frac{\theta}{2}}$ implies that $\norm{x + \ker A}_2 \geq \sin \left(\frac{\theta}{2}\right)$, which is seen as in the proof of Theorem \ref{prop:robustCrit}; for $y \in \ker A$ arbitrary, we have
	\begin{align*}
		\norm{x-y}_2^2 \geq 1 + \norm{y}_2^2 - 2 \cos\left(\frac{\theta}{2}\right)\norm{y} = \left( \norm{y}_2 - \cos \left(\frac{\theta}{2}\right)\right)^2 +1 - \cos^2\left(\frac{\theta}{2}\right)\geq \sin^2 \left(\frac{\theta}{2}\right).
	\end{align*}
	Therefore, we have for normalized $x$ with $\norm{x_T}_1 \geq \gamma \Vert x_{T^c} \Vert_1$
	\begin{align*}
		\norm{Ax}_2  \geq \sigma_{\min}(A) \norm{x + \ker A}_2 \geq \sigma_{\min}(A) \sin \left(\frac{\theta}{2}\right).
	\end{align*}
	Therefore, we can choose $\tau = \sqrt{\abs{T}}(\sigma_{\min}(A) \sin \theta)^{-1}$, since then
	\begin{align*}
		\norm{x_{T}}_1 \leq \sqrt{T} \norm{x}_2 \leq  \sqrt{\abs{T}}\left(\sigma_{\min}(A) \sin \left(\frac{\theta}{2}\right)\right)^{-1} \norm{Ax}_2 \leq \gamma \Vert x_{T^c} \Vert_1 + \tau \norm{Ax}_2,
	\end{align*}
	i.e. \eqref{eq:RNSP} is satisfied for any $x \in \R^d$.
		\end{proof}

\subsection{How Many Gaussian Measurements Are Needed to Secure the $ASC$ for a Given $\theta$?} \label{sec:GaussMeas}

In this final section we will assume that the linear map $A : \calH \to \R^m$ is Gaussian, and ask how large $m$ has to be such that $\theta$-$ASC$ for some given $\theta$ is satisfied with high probability.

Typically in compressed sensing, one can prove that as long as the parameter $m$ is larger than a certain threshold, depending on the dimension of $\calH$ and the type of structured signals, a program of the form $(\calP^\epsilon_f)$ will exhibit stability and robustness. One way of proving such results is to use Gordon's ''Escape through a mesh''-lemma \cite{Gordon1988}. This lemma relates the so called \emph{Gaussian width} $w(T)$ to the dimension that a uniformly distributed random subspace has to have in order to miss a set $T$ contained in the unit sphere $\sph(\calH)$ of $\cal H$ (which clearly is equivalent to missing $\cone(T)$). Let us make this idea more precise. If $g$ is a Gaussian vector in $\cal H$, it is defined by
	\begin{align*}
		w(T) = \mathbb{E} \bigg( \sup_{x \in T} \sprod{x,g} \bigg).
	\end{align*}
	There are several ways to state the ''Escape through a mesh''-lemma. The following one from \cite[Th. 9.21]{MathIntroToCS} is probably the most convenient for us: let $\ell_m$ denote the expected length of a Gaussian vector in $\R^m$, i.e. $\ell_m = \erw{\norm{g}_2}$ for $g \in \R^m$ Gaussian, and $T$ a subset of $\sph(\calH)$. Then we have
	\begin{align} \label{eq:meshThickEscape}
		\prb{ \inf_{x \in T} \norm{Ax}_2 \leq \ell_m - w(T) -t } \leq \exp\left(\frac{-t^2}{2}\right).
	\end{align}
	This particularly implies that if $\ell_m>  w(T)+ 2\sqrt{\eta^{-1}}$,
	\begin{align} \label{eq:meshEscape} 
		\prb{ \ker A \cap T = \emptyset} > \prb{\inf_{x \in T} \norm{Ax}_2 \geq \ell_m - w(T) -2\sqrt{\eta^{-1}} }>1-\eta.
	\end{align}
	Since $\ell_m \approx \sqrt{m}$, the estimate \eqref{eq:meshEscape} qualitatively tells us that if $m > w(T)^2$, $\ker A$ will probably miss the set $\cone(T)$. If we choose $T$ such that $\cone(T)= \calD(f, x_0) \cap \sph(\calH)$, we can relate this to exact recovery through the program $\calP_f$.
	
	\eqref{eq:meshThickEscape} suggests that it is possible to directly relate the  Gaussian width of the cone $C$ to the threshold amount of measurements needed to establish $\sigma_{C \to \R^m}(A)>0$, or equivalently the $ASC$ for some $\theta>0$, with high probability. This was discussed in detail in \cite{amelunxen2014gordon,chandrasekaran2012convex}, and we refer to those articles for more information.
	
We can however also compare the Gaussian width of $C^{\wedge \theta} \cap \sph(\calH)$ to the one of $C \cap \sph(\cal H)$ for a convex cone $C$. This is also interesting in its own right, since Gaussian widths in general are hard to estimate. We have the following result.

\begin{prop} \label{prop:gaussWidthExt}
	Let $C \sse \calH$ be a convex cone, and $\theta \in [0, \frac{\pi}{2}]$. Then we have
	\begin{align*}
		w(C \cap \sph(\calH)) \leq  w\left(C^{\wedge \theta} \cap \sph(\calH) \right) \leq \cos(\theta)w(C \cap \sph(\calH)) + \sin(\theta) \ell_\calH
	\end{align*}
	where $\ell_\calH:= \erw{\norm{g}}$ is the expected length of a Gaussian vector in $\calH$. In particular,
	\begin{align*}
		w(C\cap \sph(\calH))^2 \leq  w\left(C^{\wedge \theta}\cap \sph(\calH)\right)^2 \leq w(C\cap \sph(\calH))^2 + O\left(\sin(\theta)\dim \calH \right).
	\end{align*}
\end{prop}
\begin{proof}
	Since $C \sse C^{\wedge \theta}$, we trivially have $w(C\cap \sph(\calH)) \leq w( C^{\wedge \theta}\cap \sph(\calH))$. To prove the upper bound, notice that if a unit vector $x$ lies in $C^{\wedge \theta}$, there exists a unit vector $y \in C$ with $\sprod{x,y} \geq \cos(\theta)$, or (see Remark \ref{rem:angleMetric}) $\delta(x,y) \leq \theta$. The triangle inequality for the sphere metric $\delta$ now implies for every $g \in \calH$ 
	\begin{align*}
		\delta\left( y, \tfrac{g}{\norm{g}} \right)= \leq \delta\left( x, \tfrac{g}{\norm{g}} \right) + \delta(x,y) \leq \delta\left( x, \tfrac{g}{\norm{g}} \right)+ \theta.
\end{align*}	 
Using the monotonicity properties of $\cos$, this implies
\begin{align*}
	\sprod{x, \frac{g}{\norm{g}}}= \cos\left(\delta\left( x, \tfrac{g}{\norm{g}} \right)\right) \leq \cos\left ( \pos\left(\delta\left( y, \tfrac{g}{\norm{g}} \right) - \theta \right)\right),
\end{align*} 
where $\pos(t) = \max(t,0)$ denotes the positive part of a real number $t$. Using the cosine addition formula, this yields
\begin{align*}
\sprod{x, \tfrac{g}{\norm{g}}} &\leq \begin{cases} 1 &\text{ if } \delta\left( y, \tfrac{g}{\norm{g}} \right)<\theta \\
												 \sprod{y, \tfrac{g}{\norm{g}}} \cos(\theta) + \sin\left(\delta\left( y, \tfrac{g}{\norm{g}} \right)\right) \sin(\theta) & \text{ else}
\end{cases} \\
&\leq \cos(\theta)  \sprod{y, \tfrac{g}{\norm{g}}} + \sin(\theta),
\end{align*}
where the last step follows from $\sin\left(\delta\left( y, \frac{g}{\norm{g}} \right)\right)\leq 1$, $\sin(\theta) \in [0,1]$, $\cos(\theta) \geq 0$ (since $\theta \in [0, \frac{\pi}{2}]$) and 
\begin{align*}
	\cos(\theta)  \sprod{y, \tfrac{g}{\norm{g}}} + \sin(\theta) \geq \cos^2(\theta) + \sin^2(\theta) =1
\end{align*}
for $y$ with $\delta\left( y, \tfrac{g}{\norm{g}} \right)<\theta$. Consequently
	\begin{align*}
		w\left(C^{\wedge \theta}\cap \sph(\calH)\right) = \erw{ \sup_{x \in C^{\wedge \theta}, \norm{x}=1} \sprod{x,g}} & \leq \erw{ \sup_{y \in C, \norm{y}=1} \cos(\theta) \sprod{y,g} + \norm{g} \sin(\theta)} \\
		&= \cos(\theta)w(C\cap \sph(\calH)) + \sin(\theta)\ell_\calH.
		\end{align*}
		For the second inequality, we simply need to square the first one:
\begin{align*}
w\left(C^{\wedge \theta} \cap \sph(\calH) \right)^2 &\leq \cos^2(\theta)w(C \cap \sph(\calH))^2 +  2\cos(\theta)\sin(\theta)w(C \cap \sph(\calH)\ell_\calH + \sin^2(\theta) \ell_\calH^2 \\
&\leq w(C \cap \sph(\calH))^2 +  \left(2\cos(\theta) + \sin(\theta)\right)\sin(\theta) \ell_\calH^2 \leq  w(C \cap \sph(\calH))^2 + \sqrt{5}\sin(\theta)\ell_\calH^2,
\end{align*}		
		 where we utilized that $w(C\cap \sph(\calH)) \leq \erw{\norm{g}} =\ell_\calH = O( \sqrt{\dim \calH} )$.
	\end{proof}
	
	The $ASC$ states that if $\calD^{\wedge \theta}(f,x_0) \cap \ker A = \set{0}$, the program $(\calP_f^\epsilon)$ will stably recover $x_0$. According to above discussion, this will be the case provided $m$ is larger than $w(\mathcal{D}^{\wedge\theta}(f,x_{0}) \cap \sph(\mathcal{H}))^{2}$. The last proposition therefore shows that in order to achieve the $ASC$ for a given $\theta$, we need $O(\sin\left(\theta\right)\dim \calH)$ more Gaussian measurements than we would need to ensure \eqref{eq:geomExactCond}. Let us end by considering an example which shows that this claim in general cannot be substantially improved. For this, we will use an alternative tool to calculate thresholds as described above: the \emph{statistical dimension $\delta(C)$ of a convex cone $C$}. It was introduced in  \cite{AmelunxLotzMcCoyTropp2014}, where the authors also proved that if $m> \delta(C)$, the probability that $\ker A \cap C = \set{0}$ is very high. $\delta(C)$ is always close to $w(C \cap \sph(\calH))$ -- in fact, we have \cite[Prop 10. 2]{AmelunxLotzMcCoyTropp2014}
	\begin{align*}
		w(C)^2 \leq \delta(C) \leq w(C)^2 +1.
	\end{align*}
	
	\begin{example}
		Consider the circular cone $\text{Circ}_d(\alpha) \sse \R^d$,
		\begin{align*}
			\text{Circ}_d(\alpha) = \set{ x \in \R^d \vert x(1) \geq \cos(\alpha) \norm{x}_2}
\end{align*}		 
According to \cite[Proposition 3.4]{AmelunxLotzMcCoyTropp2014}, the statistical dimension of $\text{Circ}_d(\alpha)$ is equal to $d\sin^2(\alpha) +O(1)$. It is furthermore clear that $\text{Circ}^{\wedge \theta}_d(\alpha) = \text{Circ}_d(\alpha+\theta)$ (in particular again a convex cone). Hence,
\begin{align*}
	\delta\left(\text{Circ}^{\wedge \theta}_d(\alpha)\right) &= d \sin^2(\alpha+\theta) + O(1) = d (\sin^2(\alpha)\cos^2(\theta) + \frac{1}{2}\sin(2\alpha)\sin(2\theta) + \cos^2(\alpha) \sin^2(\theta)) +O(1).
\end{align*}
If we assume that $\theta$ is small, we have $\cos^2(\theta) \approx 1$, $\sin^2(\theta) \approx 0$ and $\sin(2\theta) \approx 4\sin\left(\theta \right)$. Hence, we obtain for such $\theta$
\begin{align*}
	\delta\left(\text{Circ}^{\wedge \theta}_d(\alpha)\right) \approx \delta( \text{Circ}_d(\alpha)) + O\left(d\sin\left(\theta\right)\right) + O(1).
\end{align*}
	\end{example}

\subsection*{Acknowledgement}
The author acknowledges support by Deutsche
Forschungsgemeinschaft (DFG) Grant KU 1446/18~-~1, as well as Grant SPP 1798, as well as the Deutscher Akademischer Austausch Dienst (DAAD), which supported him with a scholarship for his master studies, during which the research was completed. He also likes to thank Gitta Kutyniok for the supervision of the Master Thesis leading up to this article, Martin Genzel for careful proofreading and many useful suggestions for increasing the readability of the paper and Dae Gwan Lee  for pointing out a few errors, and suggesting ways to improve some statements in the paper. He also wishes to thank the anonymous reviewers, whose comments, criticisms and suggestions greatly improved the final version of this article, in particular regarding Sections \ref{sec:ASCvsConds} and \ref{sec:GaussMeas}.

\bibliographystyle{abbrv}
\bibliography{bibliographyMastersThesis}

\appendix
\section[Appendix A:]{Proof of Lemma \ref{lem:boundingAngle}} \label{app:A}

Here, as promised, we present the proof of Lemma \ref{lem:boundingAngle}.
\begin{proof}[Proof of Lemma \ref{lem:boundingAngle}]
	If we define for $x_0 \in \calC$
	\begin{align*}
		\mu(x_0) = \sup_{x \in P \backslash\set{x-x_0}} \sup_{z \in U, \norm{z}=1} \sprod{\frac{x-x_0}{\norm{x-x_0}},z},
	\end{align*}
	we need to prove that $\mu(x_0)<1$ for each $x_0 \in \calC$, and further more that $\mu(x_1)=\mu(x_2)$ if $x_1$ and $x_2$ lie in the relative interior of the same face $F$. To do this, let us begin by remarking that for each $x \in P$ and $z \in U$, $\norm{z}=1$, there must be
	\begin{align} \label{eq:smallerThanOne}
		\sprod{\frac{x-x_0}{\norm{x-x_0}},z}<1
	\end{align}
	Otherwise, the equality statement of the Cauchy-Schwarz inequality would imply that $x-x_0 = \lambda z$ for some $\lambda \in \R$, which means that $x \in x_0 +U$, which is ruled out by assumption.
	
	\eqref{eq:smallerThanOne} implies that we only need to prove that the supremum defining $\mu(x_0)$ is attained in some points $\hat{x} \in P \backslash \set{x_0}$ and $\hat{z} \in U \cap \sph_\calH$, where we introduced the notation $\sph_\calH = \set{z \in \calH \ \vert \ \norm{z}=1}$. Towards this end, let $(x_n)$ and $(z_n)$ be such that
	\begin{align} \label{eq:approxSeq}
		\sprod{\frac{x-x_0}{\norm{x-x_0}},z} \to \mu(x_0).
	\end{align}
	Since $\calH$ is finite-dimensional, and $P$ and $U \cap \sph_\calH$ are closed, there will exist $\hat{x} \in P$ and $\hat{z}\in U \cap \sph_\calH$ such that, possibly after going over to subsequences, $x_n \to \hat{x}$, $z_n \to z$. It is only left to prove that we can choose $x_n$ in such a way so that $\hat{x} \neq x_0$.
	
	So suppose $x_n \to x_0$. By the assumption of the theorem, $x_0$ is a member of some face $F$ of $P$, which by definition is defined by a set of linear equalities and inequalities. If we write $H^{-}(c, \gamma)= \set{ x \in H : \sprod{c,x} \leq \gamma}$ and $H(c, \gamma)= \set{ x \in H : \sprod{c,x} = \gamma}$, we have
	\begin{align*}
		F = \bigcap_{i=1}^k H(a_i, \alpha_i)  \cap  \bigcap_{j=1}^\ell H^{-}(b_j, \beta_j),
	\end{align*}
	say, with $P = \bigcap_{i=1}^k H^{-}(a_i, \alpha_i) \cap \bigcap_{j=1}^\ell H^{-}(b_j, \beta_j)$.  Hereby, we may without loss of generality assume that all vectors $(a_i)$ and $(b_j)$ are normalized, and furthermore that $\sprod{b_j, x_0}<\beta_j$ for all $j$ (if not, we may rename some of the $b_j$ to $a_i$'s).
	
	If we define $\tilde{x}_n = x_0 + \lambda_n v_n$ for some $v_n = x_n -x_0$ and some $\lambda_n>0$, it is not hard to convince oneself that also $\tilde{x}_n$ is a sequence that satisfies \eqref{eq:approxSeq}. If we can prove that $\lambda_n$ can be chosen in such a manner that $\norm{\tilde{x}_n - x_0} \geq \delta >0$, but still $\tilde{x}_n \in P$ for every $n$, we are done.

Due to the definition of $P$, we know that $\tilde{x}_n \in P$, if and only if  $\sprod{a_i, x_0 +\lambda_n v_n} \leq \alpha_i$ and $\sprod{b_j, x_0 +\lambda v_n} \leq \beta_j$ for each $i$ and $j$. Since $\sprod{a_i, x_0} = \alpha_i$ and $\sprod{a_i, x_n}\leq \alpha_i$, the conditions involving $a_i$'s will always be fulfilled. As for the ones involving $b_j$'s, notice that
	\begin{align*}
		\sprod{b_j, \lambda_n (x_n -x_0)} \leq \beta_j - \sprod{b_j,x_0}
	\end{align*}
	for each $j$. We assumed that $\beta_j-\sprod{b_j,x_0}>0$ for each $j$. Therefore, if $\sprod{b_j, v_n} \leq 0$ for all $j$, $\sprod{b_j, x_0 + \lambda_n v_n} \leq \beta_j$ for each possible $\lambda_n$, in particular for $\lambda_n$ so large that $\norm{\tilde{x}_n -x_0}_2 \geq \epsilon$. Otherwise we choose $\lambda_n$ in such a way so that for some $j_0$, $\sprod{b_{j_0}, \lambda_n v_n}  = \beta_{j_0} - \sprod{b_{j_0}, x_0}$. This has the consequence that
	\begin{align*}
		\norm{\lambda_n (x_n-x_0)} \geq \sprod{b_{j_0}, \lambda_n (x_n-x_0)} = \beta_{j_0}- \sprod{b_{j_0}, x_0} \geq \min_j (\beta_j - \sprod{b_j, x_0})=:\delta>0.
	\end{align*}
	For this choice of $\lambda_n$, we have $\tilde{x}_n \in P$ and $\norm{\tilde{x}_n-x_0} \geq \delta >0$ for all $n$, which we set out to prove.

	Now we will prove that $\mu(x_0)$ only depends on in which face $F$ of $P$ $x_0$ lies. Let us assume that $x_0$ and $\hat{x}_0$ both lie in the relative interior of the face $F$, which is described in the same way as above, i.e. that
	\begin{align*}
		\sprod{a_i, x_0} =\sprod{a_i, \hat{x}_0} &=\alpha_i \\
		\sprod{b_j, x_0}< \beta_j, \sprod{b_j, \hat{x}_0} &< \beta_j.  
	\end{align*}
	Now let $x$ and $z$ be such that $\sprod{(x-x_0)/ \norm{x-x_0}_2, z} = \mu(x_0)$. It is clear that $\sprod{a_i,x} \leq \alpha_i$ and $\sprod{b_j, x} \leq \beta_j$ for each $i$ and $j$, since $x \in P$. We may in fact assume that $\sprod{b_j, x}< \beta_j$ for each $j$: If there exists a $j$ with $\sprod{b_j, x} = \beta_j $, we may, by considering the segment between $x$ and $x_0$ (on which the function $x \to \sprod{(x-x_0)/\norm{x-x_0},z}$ is constant), find a vector $\tilde{x}$ with
	\begin{align*}
		\mu(x_0)=\sprod{\frac{x-x_0}{ \norm{x-x_0}}, z}= \sprod{\frac{\tilde{x}-x_0}{ \norm{\tilde{x}-x_0}}, z},
	\end{align*}
	but $\sprod{b_j,\tilde{x}}<\beta_j$, since $\sprod{b_j,x_0}<\beta_j$.
	
	Now, $\sprod{a_i, x + (\hat{x}_0 - x_0)} = \sprod{a_i, x} \leq \alpha_i$. Furthermore, if $\norm{x_0- \hat{x}_0}$ is small, we will have $\sprod{b_j, x + ( \hat{x}_0 - x_0)}~<~\beta_j$,
	since $\sprod{b_j,x}<\beta_j$. Hence, $x+ (\hat{x}_0 - x_0) \in P$ and therefore
	\begin{align*}
		\mu(\hat{x}_0) \geq \sprod{ \frac{x+(\hat{x}_0-x_0)-\hat{x}_0}{\norm{x+(\hat{x}_0-x_0)-\hat{x}_0}}, z} = \sprod{\frac{x-x_0}{\norm{x-x_0}},z} = \mu(x_0).
	\end{align*}
	A symmetric argument proves that $\mu(x_0) \geq \mu(\hat{x}_0)$. Hence, around each $x_0 \in \relint F$, there exists an open ball $U_{x_0}$ so that $\mu$ is constant in $F \cap U_{x_0}$. Since $\relint F$ is a connected set, this proves that $\mu$ is constant over the whole of $\relint F$.
\end{proof}

 \end{document}